\documentclass[11pt]{amsart} 
\usepackage[utf8]{inputenc}
\usepackage[a4paper]{geometry}
\usepackage{color}
\usepackage{amssymb}
\usepackage{mathrsfs}
\usepackage[all]{xy}
\usepackage{tikz-cd}
\usepackage{comment}
\usepackage[skip=2pt plus1pt, indent=20pt]{parskip}
\newtheorem{thm}{Theorem}[section]
\newtheorem{cor}[thm]{Corollary}
\newtheorem{lem}[thm]{Lemma}
\newtheorem{prop}[thm]{Proposition}
\theoremstyle{definition}
\newtheorem{defin}[thm]{Definition}
\newtheorem{rem}[thm]{Remark}

\numberwithin{equation}{section}

\newcommand{\vertiii}[1]{{\left\vert\kern-0.25ex\left\vert\kern-0.25ex\left\vert #1 
		\right\vert\kern-0.25ex\right\vert\kern-0.25ex\right\vert}}
		
\newcommand{\re}{\mathfrak{Re}\,}
\newcommand{\im}{\mathfrak{Im}\,}
\newcommand{\marg}[1]{\marginpar{\tiny #1}} 

\def\({\left(}
\def\){\right)}

\def\N{\mathbb{ N}}
\def\R{\mathbb{ R}}

\def\fA{\mathfrak A}
\def\fB{\mathfrak B}

\def\fc{\mathfrak c}

\newcommand{\eps}{\varepsilon}

\newcommand{\cA}{{\mathcal A}}
\newcommand{\cB}{\protect{\mathcal B}}

\newcommand{\cL}{{\mathcal L}}

\newcommand{\cP}{{\mathcal P}}

\newcommand\PS{{\normalfont\textsf{PS}}_\textsf{2}}

\newcommand{\Id}{\operatorname{Id}}

\newcommand\restr[2]{{
		\left.\kern-\nulldelimiterspace 
		#1 
		\right|_{#2} 
}}

\newcommand{\bil}[2]{{\left\langle\kern0ex #1,#2
		\kern0ex\right\rangle}}	

\def\1{\textbf{1}}

\def\fin{\operatorname{fin}}

\DeclareMathOperator{\Hom}{Hom}

\newcommand\JL{\operatorname{JL}}

\title[The CSP for Banach lattices]{A negative solution to the complemented subspace problem for Banach lattices}



\author[de Hevia]{David de Hevia}
\address{Instituto de Ciencias Matem\'aticas (CSIC-UAM-UC3M-UCM)\\
Consejo Superior de Investigaciones Cient\'ificas\\
C/ Nicol\'as Cabrera, 13--15, Campus de Cantoblanco UAM\\
28049 Madrid, Spain.}
\email{david.dehevia@icmat.es}

\author[Mart\'inez-Cervantes]{Gonzalo Mart\'inez-Cervantes}
\address{Universidad de Murcia, Departamento de Matem\'{a}ticas, Campus de Espinardo 30100 Murcia, Spain.}
\email{gonzalo.martinez2@um.es}

\author[Salguero-Alarc\'on]{\\ Alberto Salguero-Alarc\'on}
\address{Departamento de An\'alisis Matem\'atico y Matem\'atica Aplicada\\
Universidad Complutense de Madrid\\
Plaza de las Ciencias 3\\
28040 Madrid, Spain.}
\email{albsalgu@ucm.es}

\author[Tradacete]{Pedro Tradacete}
\address{Instituto de Ciencias Matem\'aticas (CSIC-UAM-UC3M-UCM)\\
Consejo Superior de Investigaciones Cient\'ificas\\
C/ Nicol\'as Cabrera, 13--15, Campus de Cantoblanco UAM\\
28049 Madrid, Spain.}
\email{pedro.tradacete@icmat.es}

\date{\today}

\begin{document}

\begin{abstract}
Building on a recent construction of G. Plebanek and the third named author, it is shown that a complemented subspace of a Banach lattice need not be linearly isomorphic to a Banach lattice. This solves a long-standing open question in Banach lattice theory.
\end{abstract}

\maketitle

\section{Introduction}

In their 1987 paper \cite{CKT}, P.G. Casazza, N.J. Kalton and L. Tzafriri started by recalling: \textit{``One of the most important problems in the theory of Banach lattices, which is still open, is whether any complemented subspace of a Banach lattice must be linearly isomorphic to a Banach lattice.''} We will refer to this as the \textit{Complemented Subspace Problem} --CSP, for short-- for Banach lattices. This could be framed in the larger research program of understanding the relation between linear and lattice structures on Banach lattices --see \cite{BVL,JMST,kalton, LT2} or \cite[Chapter 5]{Meyer} for classical results, and \cite{AMRT,ART,DLOT,LLOT} for more recent developments. Our aim here is to provide a negative solution to the CSP for Banach lattices.

This problem, although not always explicitly stated, is actually the motivation behind a considerable amount of work in the literature. It was particularly relevant in the research on local unconditional structures in Banach spaces --see \cite{BL76, FJT, GL}-- which generalize the theory of $\mathcal{L}_p$-spaces. In particular, the main conjecture studied in \cite{FJT} is whether every complemented subspace of a Banach lattice has local unconditional structure. Recall that a Banach space $X$ has local unconditional structure in the sense of Gordon-Lewis --GL-lust, for short-- if and only if $X^{**}$ is complemented in a Banach lattice --see \cite[Section 9]{JLhandbook} for details.

The difficulty of the CSP for Banach lattices lies in the fact that most existing criteria used to show that a Banach space is not linearly isomorphic to a Banach lattice do not actually distinguish between Banach lattices and their complemented subspaces. These criteria include for instance the following well-known characterization of reflexivity: if a (complemented subspace of a) Banach lattice does not contain any subspace isomorphic to $c_0$ or $\ell_1$, then it is reflexive \cite[Theorem 1.c.5]{LT2}. Similarly, weak sequential completeness can be characterized by the lack of subspaces isomorphic to $c_0$ --see \cite[Theorem 1.c.4]{LT2}. Another potential difference between Banach spaces and Banach lattices is that the latter contain plenty of unconditional sequences: every sequence of disjoint vectors in a Banach lattice is an unconditional basic sequence; similarly, every complemented subspace of a Banach lattice also contains an unconditional basic sequence \cite[Proposition 1.c.6, Theorem 1.c.9]{LT2}.

Based on these criteria one can exhibit examples of Banach spaces which are not isomorphic to complemented subspaces of Banach lattices: the James space \cite{J} and the space $Y$ constructed by Bourgain and Delbaen in \cite{BD}, since these are non-reflexive but do not contain $c_0$ nor $\ell_1$; also, any hereditarily indecomposable space as these fail to contain unconditional basic sequences \cite{GM}. Another example is the construction due to M. Talagrand --\cite{Tal}, see also \cite{LAT}-- of a reflexive space associated with a weakly compact operator $T\colon \ell_1\rightarrow C[0,1]$ which does not factor through any reflexive Banach lattice --the latter by completely different techniques to those mentioned.

Failure of local unconditional structure also allows us to add more examples to the list of spaces not isomorphic to complemented subspaces of Banach lattices: the Kalton-Peck space \cite{JLS}, Schatten $p$-class operators for $p\neq 2$ \cite{GL}, $H^\infty(\mathbb{D})$ \cite{Pelczynski}, or certain Sobolev spaces and spaces of smooth functions \cite{PW, PW2, Tse}. On the other hand, having local unconditional structure does not ensure being complemented in a Banach lattice: the examples constructed in \cite{BD} of $\mathcal{L}_\infty$-spaces without isomorphic copies of $c_0$ cannot be complemented subspaces of Banach lattices, since if this were the case, then those spaces would be complemented in their biduals \cite[Theorem 1.c.4 and Proposition 1.c.6]{LT2}, which is impossible.

An analogous version of the CSP for \textit{purely atomic} Banach lattices can also be considered, and this turns out to be equivalent to the well-known open question whether every complemented subspace of a space with an unconditional basis has an unconditional basis --which is attributed to S. Banach.

Note that under additional assumptions on the projection, the CSP for Banach lattices is known to have an affirmative solution. This is the case for example if we take a \textit{positive} projection on a Banach lattice, in which case the range is always isomorphic to a Banach lattice --see \cite[p.~214]{Schaefer-book}. As for contractive projections, it is well-known that for any $1\leq p <\infty$ every $1$-complemented subspace of an $L_p$-space is isometrically isomorphic to an $L_p$-space \cite{Bernau-LaceyLp}, separable $1$-complemented subspaces of $C(K)$-spaces are isomorphic to $C(K)$-spaces \cite{B} and, in the complex setting, N. J. Kalton and G. V. Wood proved that every $1$-complemented subspace of a space with a $1$-unconditional basis also has a $1$-unconditional basis \cite{KW} --see also \cite{Flinn}. This last fact does not hold in the real case as shown in \cite{BFL}. All these results and several others concerning $1$-complemented subspaces of Banach lattices may be found in a comprehensive survey about the topic due to B. Randrianantoanina \cite{Beata}.

A closely related long-standing open problem is the CSP for $L_1$-spaces. In the separable setting, this amounts to determining whether every complemented subspace of $L_1[0,1]$ is isomorphic to $L_1[0,1]$, $\ell_1$ or $\ell_1^n$. Significant work on this question can be found in \cite{Douglas, ES, Tal90}.

A somehow dual question is the well-known CSP for $C(K)$-spaces, which asks whether every complemented subspace of a $C(K)$-space is also isomorphic to a $C(K)$-space, with a considerable amount of literature around it --see \cite{B78, Bou, JKS, LW, Ros72}-- and a comprehensive survey due to H.P.~Rosenthal in \cite{Ros}. An explicit connection among the different versions of the CSP will be exhibited in the next section. Namely, a positive answer to the CSP for Banach lattices would imply a positive answer to the CSP for both $L_1$-spaces and $C(K)$-spaces (at least in the separable setting).

The CSP for $C(K)$-spaces has been recently solved by G. Plebanek and the third named author in \cite{PSA}, providing a non-separable counterexample by means of an appropriate construction of a Johnson-Lindenstrauss space. The main goal of the present paper is to prove that the space constructed in \cite{PSA}, which will be denoted by $\PS$ and is $1$-complemented in some $C(K)$-space, is not even isomorphic to a Banach lattice --Theorem \ref{thm:main}--, thus answering the CSP for Banach lattices also in the negative. The proof of this theorem will be essentially based on the following two facts: 
\begin{itemize}
    \item Local theory of Banach lattices can be used to show that it is enough to check that $\PS$ is not isomorphic to an AM-space --see Corollary \ref{cor:Am-spaces}. In addition, certain distinctive features of $\PS$ will allow us to further simplify the problem: it will be sufficient to prove that this space cannot be isomorphic to a sublattice of $\ell_\infty$ --Proposition \ref{prop:simplification}.
    \item We will rewrite the latter in a more convenient way --what we call \textit{having the Desired Property} in Definition \ref{def:desired-property}-- and a careful analysis of the construction of $\PS$ will show that it does actually have this property.
\end{itemize}
Subsequently, we will describe how $\PS$ can be modified to also give a negative solution to the CSP for \textit{complex} Banach lattices --Theorem \ref{thm:complex-CSP}. Moreover, this variation $\widetilde{\PS}$ is still a counterexample to the CSP for real Banach lattices --Corollary \ref{cor:final}.

The paper is organized as follows. Section \ref{SectionAuxiliaryBL} is devoted to gather some necessary results about Banach lattices. Given that our results are strongly based on the construction of $\PS$ from \cite{PSA}, for the sake of readability we summarize the basic necessary facts about $\PS$ on Section \ref{SectionTheSpace}. These will be used in Sections \ref{SectionPSnoBL} and \ref{SectionCSPcomplex} to obtain counterexamples for the CSP both for real and complex Banach lattices.

\section{Preliminaries and auxiliary results on Banach lattices}
\label{SectionAuxiliaryBL}

A Banach lattice is a real Banach space $X$, equipped with a partial order --compatible with the vector space structure-- and lattice operations $x\vee y$ and $x\wedge y$ being respectively the least upper bound, and the greatest lower bound of $x,y\in X$, so that the norm satisfies $\|x\|\leq \|y\|$ whenever $|x|\leq |y|$, where $|x|=x\vee (-x)$.

Let $X$ be a Banach lattice. A linear functional $x^* \in X^*$ is said to be a lattice homomorphism if $x^*(x\vee y)=x^*(x) \vee x^*(y)$ and $x^*(x\wedge  y)=x^*(x) \wedge x^*(y)$ for every $x,y \in X$. We denote by $\Hom(X, \R)$ the set of all lattice homomorphisms in $X^*$. The Banach space dual, $X^*$, is also a Banach lattice, with the order induced by the positive functionals, that is $x^*\geq0$ when $x^*(x)\geq0$ for every $x\geq0$. Note that every lattice homomorphism in $X^*$ is in particular positive. Fortunately enough, it suffices to know the lattice structure of $X^*$ in order to determine the linear functionals in $X^*$ that belong to $\Hom(X,\R)$. Namely, it is well-known that a functional $x^* \in X^*$ with $x^*>0$ (i.e. $x^*$ is positive and $x^*\neq 0$) is a lattice homomorphism if and only if it is an atom in $X^*$ --see, e.g., \cite[Section 2.2, Exercise 5]{AB}. Recall that an element $x>0$ in a Banach lattice $X$ is said to be an atom if and only if $x \geq u \geq 0$ implies that $u=ax$ for some scalar $a \geq 0$. 
For instance, if $X=\ell_1(\Gamma)$ for some set $\Gamma$, then it is immediate that the atoms of $X$ are precisely those elements of the form $\lambda e_\alpha$, where $\lambda>0$ and $\{e_\alpha: \alpha\in \Gamma\}$ are the vectors of the canonical basis of $\ell_1(\Gamma)$. 
Thus, for any Banach lattice $X$ with $X^*=\ell_1(\Gamma)$ we have $\Hom(X,\R)=\{\lambda e_\alpha: \lambda \geq 0,~\alpha\in \Gamma\}$.

A Banach lattice $X$ is said to be an AL-space if $\|x+y\|=\|x\|+\|y\|$ for all $x,y\in X$ with $x,y\geq 0$. Analogously, an AM-space is a Banach lattice $X$ whose norm satisfies $\|x\lor y\|=\max\{\|x\|,\,\|y\|\}$ whenever $x,y\in X$, $x,y\geq 0$. There is a well-known duality relation between these two classes: $X$ is an AL-space --respectively, an AM-space-- if and only if $X^*$ is an AM-space --resp., an AL-space-- \cite[Theorem 4.23]{AB}. The remarkable representation theorems of Kakutani allow us to identify AL-spaces with $L_1$-spaces (up to a lattice isometry), while AM-spaces can be identified in a lattice isometric way with sublattices of $C(K)$-spaces --see, for instance, \cite[Theorem 4.27, Theorem 4.29]{AB} or \cite[Theorem 1.b.2, Theorem 1.b.6]{LT2}.

Given $1\leq p\leq \infty$ and $\lambda\geq 1$, a Banach space $X$ is said to be an $\mathcal{L}_{p,\lambda}$-space if for every finite-dimensional subspace $E$ of $X$ there is a finite dimensional subspace $F$ of $X$ such that $E\subseteq F$ and $F$ is $\lambda$-isomorphic to $\ell_p^{n}$, where $n=\text{dim}\, F$. We say that a Banach space $X$ is an $\mathcal{L}_p$-space if it is an $\mathcal{L}_{p,\lambda}$-space for some $\lambda$. The most relevant properties of these classes of Banach spaces may be found in \cite{LP68, LR69}.

The discussion in the Introduction about Banach spaces being isomorphic or not to Banach lattices suggests this is not a simple question in general. It is however considerably easier to prove that a given space is not \textit{isometric} to a Banach lattice --see for example \cite[Theorem 4.1]{HHM83}, where a similar but less complicated construction than that of \cite{PSA} is used to produce a space not isometric to any Banach lattice.  

In our argument to show that $\PS$ cannot be isomorphic to a Banach lattice it will be enough to show that it cannot be isomorphic to a sublattice of $\ell_\infty$ --see Corollary \ref{cor:desired property}. The reason for this is based on the following proposition stated in \cite{AW75}; since the proof is not given explicitly there, we include one below for the convenience of the reader.

\begin{prop}\label{p:BLL1}
Let $X$ be a Banach lattice which is an $\mathcal{L}_1$-space. Then $X$ is lattice isomorphic to an $L_1$-space.
\end{prop}
\begin{proof}
Fix $\lambda>1$ such that $X$ is an $\mathcal{L}_{1,\lambda}$-space. Since $X$ is an $\mathcal{L}_1$-space, then it is isomorphic to a subspace of a certain $L_1(\mu)$ \cite[Proposition 7.1]{LP68}. Hence, $X$ cannot contain isomorphic copies of $c_0$ and, thus, $X$ is an order continuous Banach lattice \cite[Theorem 4.60]{AB}.

Let us consider the following norm in $X$:
\begin{equation}\label{eq:definition norm}
\vertiii{x}=\sup\left\{\sum_{i=0}^m \|x_i\|\::\: m\in\N,\:\, x_0,\ldots, x_m\in X \text{ pairwise disjoint s.t. } \sum_{i=0}^m x_i=x \right\}.
\end{equation}
We claim that $\vertiii{\cdot}$ defines an equivalent AL-norm for $X$. This fact implies, by Kakutani's representation theorem, that $X$ endowed with this new norm is lattice isometric to an $L_1$-space \cite[Theorem 4.27]{AB} and thus we obtain that $X$ is lattice isomorphic to an $L_1$-space.

Let us detail the proof of the previous claim. We begin by showing the next inequalities:
\begin{equation}\label{eq:equivalent norm}
\|x\|\leq \vertiii{x} \leq (K_G \lambda)^2 \|x\|, \quad \text{ for every } x\in X,
\end{equation}
where $K_G$ stands for Grothendieck's constant for real scalars. The first inequality is trivial, so we shall focus on the second one. Fix a natural number $m$ and $x_0,\ldots, x_m$ pairwise disjoint vectors in $X$. For each $i\in\{0,\ldots, m\}$, let $B_i$ be the band generated by $x_i$ in $X$. Since $X$ is order continuous, in particular, $X$ is $\sigma$-complete and thus each $B_i$ is a projection band \cite[Proposition 1.2.11]{Meyer}. Let us denote by $P_i$ its corresponding band projection from $X$ onto $B_i$ for $0\leq i \leq m$. For each $0\leq i \leq m$, take $x_i^*\in S_{X^*}$ such that $x_i^*(x_i)=\|x_i\|$. 


Consider the bounded linear operator
\begin{eqnarray*}
S&: &X\longrightarrow \:\ell_2^{m+1} \\
&&x\,\mapsto (x_i^* \circ P_i(x))_{i=0}^m
\end{eqnarray*}
Since $X$ is an $\mathcal{L}_{1,\lambda}$-space, we have that $S$ is a $1$-summing operator with $\pi_1(S)\leq K_G \lambda \|S\|$ \cite[Theorem 3.1]{DJT}. Now, we are going to show that $\|S\|\leq K_G \lambda$. To this end, we introduce the following family of operators, which are defined for each $x\in X$ by the linear extension of
\begin{eqnarray*}
T_x& :&\ell_\infty^{m+1} \longrightarrow X \\
&&\:e_i \longmapsto P_i(x)
\end{eqnarray*}
Using again the fact that $X$ is an $\mathcal{L}_{1,\lambda}$-space, by \cite[Theorem 3.7]{DJT} we have that $T_x$ is $2$-summing with $\pi_2(T_x)\leq K_G\lambda \|T_x\|$ for every $x\in X$.
Taking into account that the elements $(P_i(x))_{i=0}^m$ are pairwise disjoint and that each $P_i$ is a band projection, we obtain that
$$ \left| T_x\bigl((a_i)_{i=0}^m\bigr) \right|=\left|\sum_{i=0}^m a_i P_i(x) \right|=\sum_{i=0}^m |a_i||P_i(x)| 
\leq |x| \bigl\|(a_i)_{i=0}^m \bigr\|_\infty$$
and, thus, $\|T_x\|\leq \|x\|$.

Now, observe that 
$$\sup\left\{\left(\sum_{i=0}^m  |y^*(e_i)|^2\right)^{\frac{1}{2}}\::\: y^*\in B_{(\ell_\infty^{m+1})^*}\right\}=\sup\left\{\left(\sum_{i=0}^m  |a_i|^2\right)^{\frac{1}{2}}\::\: (a_j)_{j=0}^m\in B_{\ell_1^{m+1}}\right\}= 1,$$
and since $T_x$ is $2$-summing with $\pi_2(T_x)\leq K_G\lambda \|x\|$, for every $x\in X$, the next inequality holds:
$$\left(\sum_{i=0}^m \|P_i(x)\|^2\right)^{\frac{1}{2}}=\left(\sum_{i=0}^m \|T_x(e_i)\|^2\right)^{\frac{1}{2}}\leq K_G\lambda \|x\|, \qquad \text{ for every } x\in X.$$
From the above identity, it follows that
$$\|Sx\|=\left(\sum_{i=0}^m |x_i^*(P_i(x))|^2\right)^{\frac{1}{2}}\leq \left(\sum_{i=0}^m \|P_i(x)\|^2\right)^{\frac{1}{2}}  \leq K_G\lambda \|x\|, \qquad \text{ for every } x\in X,$$
so we get that $\pi_1(S)\leq (K_G\lambda)^2$. Given that the vectors $(x_i)_{i=0}^m$ are pairwise disjoint and $\sum_{i=0}^m x_i=x$, for every $x^*\in B_{X^*}$ we have that
$$\sum_{i=0}^m|x^*(x_i)|\leq \sum_{i=0}^m|x^*|(|x_i|)=|x^*|\left(\sum_{i=0}^m |x_i|\right)=|x^*|(|x|).$$
Therefore, $\sup\left\{\sum_{i=0}^m|x^*(x_i)|\::\: x^*\in B_{X^*}\right\}\leq \|x\|$ and we finally obtain
$$\sum_{i=0}^m \|x_i\|=\sum_{i=0}^m \|Sx_i\|\leq (K_G \lambda)^2\|x\|.$$
Since the preceding inequality does not depend on the choice of $m\in\mathbb{N}$ and $x_0,\ldots,x_m\in X$, this proves the identity (\ref{eq:equivalent norm}).

It is straightforward to check that the map $\vertiii{\cdot}$ defined in (\ref{eq:definition norm}) is a norm on $X$ and is complete because it is equivalent to the complete norm $\|\cdot\|$, as we have already exhibited in (\ref{eq:equivalent norm}). It remains to show that it is indeed a lattice norm. To this end, take $x,y\in X$ such that $|x|\leq |y|$ and fix a natural number $m$ and a finite sequence $(y_i)_{i=0}^m$ of pairwise disjoint vectors in $X$ such that $y=\sum_{i=0}^m y_i$. By the Riesz Decomposition Property --see \cite[Theorem 1.13]{AB}--, there exist $x_0,\ldots, x_m\in X$ satisfying $x=\sum_{i=0}^m x_i$ and $|x_i|\leq |y_i|$ for each $i=0,\ldots, m$. Thus, the vectors $(x_i)_{i=0}^m$ are also pairwise disjoint, and since $\|\cdot\|$ is a lattice norm, we have $\|x_i\|\leq \|y_i\|$ for each $i=0,\ldots, m$. This implies that $\vertiii{x}\leq \vertiii{y}$. Finally, it is easy to check that $\|x+y\|=\|x\|+\|y\|$ for every disjoint pair $x,y\in X$, so we conclude that $\vertiii{\cdot}$ is an AL-norm. By Kakutani's Theorem this finishes the proof.
\end{proof}

\begin{cor}\label{cor:Am-spaces}
Let $X$ be a Banach lattice which is an $\mathcal{L}_\infty$-space. Then $X$ is lattice isomorphic to an AM-space. 
\end{cor}
\begin{proof}
Since $X$ is an $\mathcal{L}_\infty$-space, its dual $X^*$ is an $\mathcal{L}_1$-space \cite[Theorem III (a)]{LR69}. By the previous proposition, $X^*$ is lattice isomorphic to an $L_1$-space and, thus, $X^{**}$ is lattice isomorphic to a $C(K)$-space. Since $X$ is a sublattice of $X^{**}$ \cite[Proposition 1.4.5 (ii)]{Meyer}, $X$ is lattice isomorphic to a sublattice of a $C(K)$-space, which is an AM-space.
\end{proof}

\begin{cor}\label{cor:L1}
Let $X$ be a complemented subspace in an $L_1$-space. If $X$ is isomorphic to a Banach lattice, then it is isomorphic to an $L_1$-space.
\end{cor}
\begin{proof}
Since $X$ is complemented in an $L_1$-space, by \cite[Theorem III (b)]{LR69} we get that $X$ is an $\mathcal{L}_1$-space. If $X$ is also isomorphic to a Banach lattice, the preceding proposition ensures that $X$ is isomorphic to an $L_1$-space.
\end{proof}

\begin{cor}\label{cor:CK}
Let $X$ be a separable complemented subspace in a $C(K)$-space. If $X$ is isomorphic to a Banach lattice, then it is isomorphic to a $C(K)$-space.
\end{cor}
\begin{proof}
 By \cite[Theorem 3.2]{LR69} complemented subspaces of $C(K)$-spaces are $\mathcal{L}_\infty$-spaces. If $X$ is also isomorphic to a Banach lattice, Corollary \ref{cor:Am-spaces} shows that $X$ is isomorphic to an AM-space. The conclusion follows from the fact that separable AM-spaces are isomorphic to $C(K)$-spaces \cite{B}.
 \end{proof}

\begin{rem}\label{rem:ramifications}
Corollaries \ref{cor:L1} and \ref{cor:CK} imply that a positive answer to the CSP for Banach lattices in the separable setting would also yield a positive answer to the CSP for $L_1$-spaces and for $C(K)$-spaces in the separable setting.
\end{rem}

\begin{rem}
Proposition \ref{p:BLL1} can also be used to show that if a Banach lattice $X$ is linearly isomorphic to $\ell_1$, then it must be lattice isomorphic to $\ell_1$ --see also \cite{AW75}. Note that this does not extend to isometries, in the following sense: A Banach lattice linearly isometric to $\ell_1$ need not be lattice isometric. In fact, the proof of Proposition \ref{p:BLL1} tells us that if a Banach lattice $X$ is linearly isometric to $\ell_1$, then it is $K_G^2$-lattice isomorphic to $\ell_1$ --where $K_G$ is Grothendieck's constant. Moreover, in the 2-dimensional setting it can be checked that the isomorphism constant is sharp: take $\ell_\infty^2$ which is linearly isometric to $\ell_1^2$, and one can check that any lattice isomorphism will give constant at least 2 --which coincides with the square of Grothendieck's constant for dimension 2 \cite{kri}.
\end{rem}

\begin{rem}
It would be natural to wonder whether Proposition \ref{p:BLL1} and Corollary \ref{cor:Am-spaces} can be extended to $\mathcal{L}_p$-spaces for $p\in[1,+\infty]$. This is however not the case, since for $p\in(1,+\infty)\setminus\{2\}$, $\ell_p\oplus_p \ell_2$ is a Banach lattice which is also an $\mathcal{L}_p$-space, but it is not isomorphic --even as a Banach space-- to any $L_p$-space \cite[Example 8.2]{LP68}.
\end{rem}

\section{The Plebanek-Salguero space}
\label{SectionTheSpace}

Let us denote by $\omega=\{0,1,2,...\}$ the set of non-negative integers, and write $\fin(\omega)$ for the family of all finite subsets of $\omega$. Given $\mathcal F \subseteq \cP(\omega)$, $[\mathcal F]$ denotes the smallest  Boolean subalgebra of $\cP(\omega)$ containing $\mathcal F$.

\par A family $\cA$ of infinite subsets of $\omega$ is \emph{almost disjoint} whenever $A\cap B$ is finite for every distinct $A, B\in \cA$.
Every almost disjoint family $\cA$ gives rise to a \emph{Johnson-Lindenstrauss} space  $\JL(\cA)$, which is the closed linear span inside $\ell_\infty$ of the set of characteristic functions $\{1_n: n\in \omega\} \cup \{1_A: A\in \cA\}\cup \{1_\omega\}$, where $1_n$ represents $1_{\{n\}}$.  
Alternatively, let us write $\fA = [\fin(\omega) \cup \cA]$. Then $\JL(\cA)$ is precisely the closure in $\ell_\infty$ of the subspace $s(\fA)$ consisting of all \emph{simple} $\fA$-measurable functions; that is, functions of the form $f=\sum_{i=1}^n a_i\cdot 1_{B_i}$, where $n\in \omega$, $a_i\in \R$ and $B_i\in \fA$. 

\par It is easy to check that $\JL(\cA)$ is isometrically isomorphic to a $C(K)$-space --see \cite[Theorem 1.b.6]{LT2}. The underlying compact space can be realized as the Stone space consisting of all ultrafilters of $\fA$. More explicitly, we can define 
\[ K_\cA = \omega \cup \{p_A: A\in \cA\} \cup \{\infty\}\]
and specify a topology on $K_\cA$ as follows:
\begin{itemize}
    \item points in $\omega$ are isolated;
    \item given $A\in \cA$, a basic neighbourhood of $p_A$ is of the form $\{p_A\}\cup A\setminus F$, where $F\in \fin(\omega)$;
    \item $K_\cA$ is the one-point compactification of the locally compact space $\omega \cup \{p_A: A\in \cA\}$. 
\end{itemize}
The compact space $K_\cA$ is often referred to as the \emph{Alexandrov-Urysohn compact space} associated with $\cA$. It is a separable, scattered compact space with empty third derivative. 
Please observe that $\JL(\cA)$ coincides with the subspace $\{f|_\omega: f\in C(K_\cA)\}$ of $\ell_\infty$. 

\par On the other hand, the dual of $\JL(\cA)$ is isometrically isomorphic to the space $M(\fA)$ of real-valued \emph{finitely} additive measures on $\fA$. Indeed, every $\mu\in M(\fA)$ defines a functional on $s(\fA)$ by means of integration \cite[III.2]{DS}, and every functional on $\JL(\cA)$ arises in this way. Let us recall that the norm of any measure $\nu\in M(\fA)$ is given by $\|\nu\| = |\nu|(\omega)$, where the \emph{variation} $|\nu|$ is defined as
\[ |\nu|(A) = \sup\{|\nu(B)| + |\nu(A\setminus B)|: B\in \fA, B\subseteq A\}. \]
In particular, since $\JL(\cA)$ is isometrically isomorphic to $C(K_\cA)$, and $K_\cA$ is scattered, $M(\fA)$ is isometrically isomorphic to $\ell_1(K_\cA)=\overline{\text{span}}\{\delta_n,\delta_{P_A},\delta_\infty\::\:n\in\omega,\,A\in\mathcal{A}\} \equiv \ell_1(\omega) \oplus_1 \ell_1(\cA) \oplus_1 \R$.
Therefore, every $\nu\in M(\fA)$ can be decomposed as $\nu = \mu + \bar\nu$, where $\mu$ is supported on $\omega$ and $\bar\nu$ is an element of $M(\fA)$ which vanishes on finite sets of $\omega$.

\par The spaces $\JL(\cA)$, originally introduced in \cite{JL1974}, have recently found use in Banach space theory as counterexamples or as a tool to produce them --see for instance \cite{AMR, KL, MP, PSA1}. They were in particular used in \cite{PSA} to obtain a negative solution for the complemented subspace problem for $C(K)$-spaces. 

\subsection{General facts}
\label{subsec:general}
\par The approach in \cite{PSA} is to construct two almost disjoint families $\cA, \cB\subseteq \cP(\omega\times 2)$ such that the corresponding Johnson-Lindenstrauss spaces enjoy the following properties, as stated in \cite[Theorem 1.3]{PSA}:
\begin{itemize}
    \item $\JL(\cB)$ is isomorphic to $\JL(\cA)\oplus \PS$, where both $\JL(\cA)$ and $\PS$ are isometric to $1$-complemented subspaces of $\JL(\cB)$.
    \item $\PS$ is not isomorphic to any $C(K)$-space.
\end{itemize}

\par Let us describe how such families $\cA$ and $\cB$ are defined and how they interact with each other. For this, we will work in the countable set $\omega \times 2$ rather than in $\omega$. Let us say that a subset $C\subseteq \omega \times 2$ is a \emph{cylinder} if it is of the form $C=C_0\times 2$ for some $C_0\subseteq \omega$. Given $n\in \omega$, let us write $c_n = \{n\}\times 2$. A partition $B^0, B^1$ of a cylinder $C=C_0\times 2$ \emph{splits} $C$ (or is a \emph{splitting} of $C$) if for every $n\in C_0$, the sets $B^0 \cap c_n$ and $B^1 \cap c_n$ are singletons.

\par Consider two almost disjoint families $\cA$ and $\cB$ such that: 
 \begin{itemize}
     \item $\cA = \{A_\xi: \xi < \fc\}$ is a family of cylinders in $\omega \times 2$.
     \item $\cB = \{B_\xi^0, B_\xi^1: \xi<\fc\}$ satisfies that the pair $B_\xi^0$, $B_\xi^1$ is a splitting of $A_\xi$ for $\xi<\fc$.
 \end{itemize}   
In this context, we will adopt the slight abuse of notation of \cite[Section 3]{PSA} by declaring $\JL(\cA)$ to be the closed subspace of $\ell_\infty(\omega \times 2)$ spanned by  $\{c_n: n\in \omega\} \cup \{1_A: A\in \cA\}\cup \{1_{\omega \times 2}\}$; that is to say, in the definition of $\JL(\cA)$ we consider only finite \textsl{cylinders} instead of all finite subsets of $\omega \times 2$. With these considerations, it is straightforward to see that $\JL(\cA)$ sits inside $\JL(\cB)$ as the subspace formed by all functions of $\JL(\cB)$ which are constant on cylinders, and the map 
$P: \JL(\cB) \to \JL(\cB)$ defined as 
\[ Pf(n,0) = Pf(n,1) = \frac{1}{2}\big(f(n,0)+f(n,1)\big)\]
is a norm-one projection whose image is $\JL(\cA)$ \cite[Proposition 3.1]{PSA}. Let us write $X=\ker P$, so that we have $\JL(\cB) = \JL(\cA) \oplus X$. Then the map $Q = \Id_{\JL(\cB)}- P$ acts as
 \[ Qf(n,0) = - Qf(n,1) = \frac12\big(f(n,0) - f(n,1)\big),\]
and so it is a norm-one projection onto $X$. Therefore both $X$ and $\JL(\cA)$ are isometric to $1$-complemented subspaces of $\JL(\cB)$, and the space $X$ can be defined as follows: 
\begin{equation} \label{eq:Csigma} X = \{f\in \JL(\cB): f(n,0)=-f(n,1) \ \text{ for all } n\in \omega\}. \end{equation}

\par In order to ensure that $X$ is not isomorphic to a $C(K)$-space, the families $\cA$ and $\cB$ will be chosen to satisfy certain delicate combinatorial properties. Actually, the space which we denote by $\PS$ is such a space $X$ for a particular choice of $\cA$ and $\cB$ such that $X$ is not a $C(K)$-space.

\subsection{Norming and free subsets}\label{s:normingfree}
\par We now focus on how to produce the almost disjoint families $\cA$ and $\cB$, intertwined as in Section \ref{subsec:general}, so that the resulting space $X$ is not isomorphic to a $C(K)$-space. These techniques will be essential for our counterexample below. We start with the basic idea.

\label{subsec:norming}
\begin{defin} Given a Banach space $X$ and a weak$^*$ closed subset $K$ of $B_{X^*}$, we say that
\begin{itemize}
	\item $K$ is \emph{norming} for $X$ if there is $0<c\leq 1$ such that $\sup_{x^*\in K} |\bil{x^*}{x}|\geq c\, \|x\|$ for every $x\in X$.
	\item $K$ is \emph{free} if for every $f\in C(K)$ there exists $x\in X$ such that $f(x^*)=\bil{x^*}{x}$ for every $x^*\in K$.
\end{itemize}
\end{defin}
We will speak of a \emph{$c$-norming} set for $X$ whenever we need explicit mention of the constant $c$ in the definition of a norming set. Observe that $K$ is a norming and free subset of $B_{X^*}$ precisely when the natural operator $T: X \to C(K)$ defined by $T(x)(x^*)=\bil{x^*}{x}$ is an (onto) isomorphism \cite[Lemma 2.2]{PSA1}. 

\par The use of norming free sets makes it possible to construct certain Banach spaces which are not $C(K)$-spaces. Applied to our particular case, the idea is to prevent every candidate for a norming set in the dual ball of $X$ to be free. We now indicate how to proceed. First, set $\fB := [\fin(\omega) \cup \cB]$ and observe that, for any choice of the families $\cA$ and $\cB$ as in Section \ref{subsec:general}, $X^*$ can be isometrically identified \textit{in a canonical way} with the subspace $\JL(\cA)^\perp=\{\nu\in JL(\cB)^*=M(\fB)\::\: \left.\nu\right|_{\JL(\cA)}=0\}$ of $M(\fB)$, that is, $\JL(\cA)^\perp$ is formed by all measures of $M(\fB)$ which vanish on every cylinder. Indeed, let us consider the operator $T:\JL(\cA)^\perp\to X^*$ defined by $T\nu:=\left.\nu\right|_X$. Given $x^*\in X^*$, we define $\nu=Q^*x^*\in\JL(\cB)^*$, which satisfies that for every $f\in\JL(\cA)$, $\nu(f)=Q^*x^*(f)=x^*(Qf)=x^*(0)=0$ (and so $\nu\in\JL(\mathcal{A})^\perp$) and $\left.\nu\right|_X=x^*
$ given that $Qf=f$ for all $f\in X$). The latter shows that $T$ is surjective and it remains to show that $T$ is norm-preserving. Given any $\nu \in \JL(\cA)^\perp$, we have
\begin{eqnarray*}
\|T\nu\|&=&\sup_{f\in B_X}\|\nu(f)\|\leq \sup_{f\in B_{\JL(\cB)}}\|\nu(f)\|=\sup_{f\in B_{\JL(\cB)}}\|\nu(Pf)+\nu(Qf)\| \\
&=&\sup_{f\in B_{\JL(\cB)}}\|\nu(Qf)\|\leq \sup_{f\in B_X}\|\nu(f)\|=\|T\nu\|,    
\end{eqnarray*}
so $\|T\nu\|=\sup_{f\in B_{\JL(\cB)}}\|\nu(f)\|=\|\nu\|$.

Therefore, every functional $\nu \in X^*\equiv \JL(\cA)^\perp\subseteq \JL(\cB)^*\equiv \ell_1(\omega\times 2)\oplus_1 \ell_1(\mathcal{B})\oplus_1\mathbb{R}$
can be seen as a pair of measures $(\mu, \bar\nu) \in \ell_1(\omega\times 2) \oplus_1 \ell_1(\fc\times 2)$ where $\mu(n,0)=-\mu(n,1)$ for every $n\in \omega$ and $\bar\nu(\xi,0)=-\bar\nu(\xi,1)$ for every $\xi<\fc$, where here we are using the notation $\mu(n,i):=\mu\bigl(\{(n,i)\}\bigr)$ and $\bar\nu(\xi,i):=\bar\nu(B_\xi^i)$, for $i\in\{0,1\}$.

On the other hand, for every $n\in \omega$, the function
\begin{equation}\label{eq:fn's}
 f_n = 1_{(n,0)}-1_{(n,1)}    
\end{equation}
is always a norm-one element of $X$. Hence, every norming set for $X$ must contain a sequence of functionals  $(\nu_n)_{n\in \omega}$ such that $\inf_{n\in \omega} |\nu_n(n,0)|=\frac{1}{2}\inf_{n\in\omega} |\nu_n(f_n)|>0$. This motivates the following definition:

\begin{defin}\label{def:admissible-real} A bounded sequence $(\mu_n)_{n\in \omega}$ in $\ell_1(\omega\times 2)$ is \emph{admissible} if 
\begin{itemize}
    \item $\inf_{n\in \omega} |\mu_n(n,0)|>0$.
    \item $\mu_k(n,0)=-\mu_k(n,1)$ for every $k,n\in \omega$.
\end{itemize}
\end{defin}
\begin{rem}\label{rem:sequences-to-code}
Consequently, every norming set for $X$ must contain a sequence $(\nu_n)_{n\in\omega}=(\mu_n, \bar\nu_n)_{n\in \omega}\in\ell_1(\omega\times 2)\oplus_1\ell_1(\mathfrak{c}\times 2)$ such that:
\begin{enumerate}
    \item[\emph{i)}] $(\mu_n)_{n\in \omega}$ is admissible;
     \item[\emph{ii)}] $\bar\nu_n(\alpha,0)+\bar\nu_n(\alpha,1)=0$ for every $n\in\omega$ and $\alpha<\fc$;
      \item[\emph{iii)}] Since for every natural $n$, $|\bar{\nu_n}|(B)>0$ for at most countably many $B\in\mathcal{B}$, then there exists $\xi<\mathfrak{c}$ such that $\bar\nu_n(\alpha,i)=0$ for all $\alpha\geq \xi$, $n\in\omega$ and $i\in\{0,1\}$.
\end{enumerate}
\end{rem}

The main idea of \cite{PSA} is to prevent every sequence $(\nu_n)_{n\in\omega}$ contained in $B_{\JL(\cA)^\perp}$ of the form described in the preceding remark from lying inside a free subset in the dual unit ball of $\PS$. In this way, there are no norming free subsets for $\PS$, and therefore it cannot be isomorphic to a $C(K)$-space. 

\par Our subsequent argument also makes use of admissible sequences to prove that $\PS$ is, in fact, not isomorphic to a Banach lattice. This proof relies on the following simple observation:

\begin{lem}{\cite[Lemma 3.3]{PSA}}\label{lem:no-free} Assume $\fB$ is a Boolean subalgebra of $\cP(\omega)$. If $M\subseteq M_1(\fB)=B_{\JL(\cB)^*}$ lies inside a free subset, then for every $B\in \fB$ and $\eps>0$ there is $g\in s(\fB)$ such that $\big|\bil{\mu}{g} - |\mu(B)|\big|<\eps$ for every $\mu\in M$.
\end{lem}

\subsection{Separation of measures}
Let $M_1(\fB)$ denote the unit ball of the space $M(\fB)$.
\label{subsec:separation}
\begin{defin} \label{def:sep} Given  a Boolean subalgebra $\fB$ of $\cP(\omega\times 2)$, two subsets of measures $M, M'\subseteq M_1(\fB)$ are \emph{$\fB$-separated} if there is $\eps>0$ and a finite collection $B_1, ..., B_n\in \fB$ such that for every pair $(\mu, \mu')\in M\times M'$, there is $k\in\{1,...,n\}$ such that $|\mu(B_k)-\mu'(B_k)|\geq \eps$.
\end{defin}

\par The notion of $\fB$-separation is essential in the construction of $\PS$. In particular, Definition \ref{def:sep} in tandem with Lemma \ref{lem:no-free} is what prevents a certain sequence of measures from lying inside a free set. We will also need the following fact about $\fB$-separation  to show that $\PS$ is not isomorphic to a Banach lattice: 
\begin{lem}{\cite[Lemma 4.2]{PSA}} \label{lem:sep}
Let $M,\, M'\subseteq M_1(\fB)$ be two sets of measures.
    If there exist $\eps>0$ and a simple $\fB$-measurable function $g$ such that for every $(\mu, \mu')\in M\times M'$ we have $|\bil{\mu}{g}-\bil{\mu'}{g}|\geq \varepsilon$, then $M$ and $M'$ are $\fB$-separated. 
\end{lem}

\subsection{The heart of the construction of $\PS$}
\label{subsec:heart}

The almost disjoint families $\cA$ and $\cB$ are constructed through an inductive process of length $\fc$ which is explained in \cite[Section 6]{PSA}. Let us now describe this process paying special attention to the properties that will be needed later to show that $\PS$ is not isomorphic to a Banach lattice. 
\par Recall that the idea is to construct a family $\cA = \{A_\xi: \xi <\fc\}$ of cylinders in $\omega \times 2$ and define suitable splittings $B_\xi^0, B_\xi^1$ of $A_\xi$ for every $\xi<\fc$. Given any $\Lambda\subseteq \fc$, we will denote $\fB(\Lambda) = [\fin(\omega \times 2) \cup \{B_\alpha^0, B_\alpha^1\colon \alpha\in \Lambda\}]$. In particular, $\fB(\xi)$ stands for $\fB(\{\alpha: \alpha<\xi\})$, and the final algebra is denoted $\fB=[\fin(\omega \times 2) \cup\{B_\alpha^0, B_\alpha^1\colon \alpha<\fc\}]$.

\par Let us explain how the sets $B_\xi^0, B_\xi^1$ are obtained for any given $\xi<\fc$. First, observe that we can ``code" \textit{all} sequences in $M_1(\fB)$ of the form detailed in Remark \ref{rem:sequences-to-code}, $(\nu_n^\xi)_{n\in\omega}=(\mu_n^\xi, \bar\nu_n^\xi)_{n\in \omega}$ for $\xi<\fc$, in such a way that $\bar\nu_n^\xi(\alpha,i)=0$ for all $\alpha\geq\xi$, $i=0,1$ and every $n\in\omega$. 
Moreover, every sequence $(\nu_n)_{n\in \omega}$ in the unit ball of $\ell_1(\omega \times 2) \oplus_1 \ell_1(\fc \times 2)$ which satisfies properties \emph{i)}--\emph{iii)} of the aforementioned remark is of the form $(\nu_n^\xi)_{n\in\omega}$ for exactly one $\xi<\fc$.


\par The sets $B_\xi^0$ and $B_\xi^1$ are defined, together with three infinite auxiliary subsets $J_2^\xi \subseteq J_1^\xi \subseteq J_0^\xi \subseteq \omega$, so that the sequence $(\nu_n^\xi)_{n\in\omega}$ cannot eventually lie in a free set of $M_1(\fB)$. This is done as follows. First, let us define $c=\inf_{n\in \omega}|\mu_n^\xi(n,0)|$, which is a strictly positive number, and let $p_n^\xi$ be the one element subset of $c_n=\{(n,0),(n,1)\}$ for which $\mu_n^\xi(p_n^\xi)>0$. Now, consider three infinite subsets $J_2^\xi \subseteq J_1^\xi \subseteq J_0^\xi \subseteq \omega$ such that the differences $\omega\setminus J_0^\xi$, $J_0^\xi\setminus J_1^\xi$ and $J_1^\xi\setminus J_2^\xi$ are also infinite, and in such a way that the following assertions are verified for some fixed $\delta \in (0, \frac{c}{16})$:

\begin{enumerate}
    \item[(P1)] For every  $n\in J_0^\xi$, $|\mu_n^\xi|\bigl( (J_0^\xi \times 2)\setminus c_n\bigr)<\delta$ --this is exactly (5.a) in the proof of \cite[Lemma 5.3]{PSA}.
	\item[(P2)] There is $a\geq c$ such that $|\mu_n^\xi(p_n^\xi)-a|<\delta$ for every $n\in J_1^\xi$ --this is (5.b) in the proof of \cite[Lemma 5.3]{PSA}.
    \item[(P3)] For any $\alpha < \xi$, the pairs
	\begin{itemize}
		\item $\{\nu_n^\alpha: n\in J^\alpha_2\}$ and $\{\nu_n^\alpha: n\in J^\alpha_1\setminus J^\alpha_2\}$,
		\item $\{\nu_n^\alpha: n\in J^\alpha_1\}$ and $\{\nu_n^\alpha: n\in J^\alpha_0 \setminus J^\alpha_1\}$ 
	\end{itemize}
	are not $\fB(\xi \setminus \{\alpha\})$-separated --this is exactly the Key Property in \cite[p. 16]{PSA}.
\end{enumerate}

The justification of the existence of such a trio of sets can be found in \cite[p. 16]{PSA}. We also remark the fact that although the computations in \cite[Lemma 5.3]{PSA} require that $0<\delta<c/16$, in the proof of our main Theorem \ref{thm:main} we will only need to assume that $\delta < c/11$. In any case, with the sets $J^\xi_0, J_1^\xi$ and $J_2^\xi$ in our power, we finally declare:
\[ B_\xi^0 = \left(\bigcup_{n\in J^\xi_2} p_n^\xi \right) \cup \left(\bigcup_{n\in J^\xi_1 \setminus J^\xi_2} c_n \setminus p_n^\xi \right), \quad B_\xi^1 = (J_1^\xi \times 2) \setminus B_\xi^0.\] 

\par To conclude with the construction, one needs to ensure that, as a consequence of (P1)--(P3), $(\nu_n^\xi)_{n\in\omega}$ cannot lie inside a free subset of $M(\fB)$. This is taken care of in \cite[Lemmata 5.3 and 5.5]{PSA}. Our proof that $\PS$ is not isomorphic to a Banach lattice is also substantially based on Properties (P1)--(P3). We now record two observations which will pave the way in the next Sections. 

\begin{rem} \label{rem:final}
The final algebra $\fB$ is such that (P3) is satisfied for every $\xi<\fc$. This immediately implies that, given any $\xi<\fc$, the pairs 	\begin{itemize}
		\item $\{\nu_n^\xi: n\in J^\xi_2\}$ and $\{\nu_n^\xi: n\in J^\xi_1\setminus J^\xi_2\}$,
		\item $\{\nu_n^\xi: n\in J^\xi_1\}$ and $\{\nu_n^\xi: n\in J^\xi_0 \setminus J^\xi_1\}$ 
	\end{itemize}
are not $\fB(\fc\setminus\{\xi\})$-separated. 
\end{rem}

In the sequel, the following remark will be applied to functions of the form $f = 1_{B_0^\xi} - 1_{B_1^\xi}$ for a given $\xi<\fc$.
\begin{rem} \label{rem:zero} Fix any $\xi<\fc$ and consider the sequence $(\nu_n^\xi)_{n\in \omega}$ in $M_1(\fB)$. Since $\bar\nu_n^\xi(\alpha, i)=0$ for every $\alpha \geq \xi$ and $i\in\{0,1\}$, we have $|\bar\nu_n^\xi|(B_\alpha^i) = 0$ whenever $\alpha\geq \xi$ and $i\in \{0,1\}$. Hence, for any $n\in \omega$, $\nu_n^\xi$ agrees with $\mu_n^\xi$ inside any set $B_\alpha^i$ for $\alpha \geq \xi$ and $i\in \{0,1\}$. In particular, if $f \in \PS$ has its support contained in the set $B^0_\xi \cup B^1_\xi$, then $\langle \nu^\xi_n, f \rangle = \langle \mu^\xi_n,f \rangle$ for every $n\in \omega$. 
\end{rem}

\section{The Plebanek-Salguero space is not a Banach lattice}
\label{SectionPSnoBL}

\par We now combine the results in Section \ref{SectionAuxiliaryBL} with the fundamental properties of the space $\PS$ to show that it is not linearly isomorphic to any Banach lattice. 
First of all, notice that $\PS^*$ is a $1$-complemented subspace of $\JL(\cB)^*\equiv C(K_\cB)^*\equiv\ell_1(K_\cB)$ and therefore linearly isometric to $\ell_1(\Gamma)$ for some $\Gamma$.
Furthermore, since the set $\{\delta_{(n,0)}, \delta_{(n,1)}:n\in \omega\}$ is $1$-norming for $\JL(\cB)$,  just taking the restrictions to $\PS$ we deduce that $\PS$ also has a countable $1$-norming set. These two characteristics of $\PS$ will make easier to prove that this space cannot be isomorphic to a Banach lattice, as the next proposition shows. 

\begin{prop}\label{prop:simplification}
Let $X$ be an isomorphic predual of $\ell_1(\Gamma)$ which has a countable norming set. If $X$ is isomorphic to a Banach lattice, then it is isomorphic to a sublattice of $\ell_\infty$.
\end{prop}
\begin{proof}
Let $Y$ be a Banach lattice which is isomorphic to $X$. Since $X$ is a predual of $\ell_1(\Gamma)$, $Y$ is an $\mathcal{L}_\infty$-space. Hence, by Corollary \ref{cor:Am-spaces}, we may assume that $Y$ is an AM-space. Then, $Y^*$ is an AL-space isomorphic to $\ell_1(\Gamma)$, so $Y^*$ is, in fact, lattice isometric to $\ell_1(\Gamma)$ --see \cite[Corollary to Theorem 3 of Section 15 and Theorem 4 of Section 18]{Lacey-book}. Let us denote by $(e_\gamma^*)_{\gamma\in\Gamma}$ the canonical basis of $\ell_1(\Gamma)$ and let $(y_n^*)_{n\in\omega}$ be a countable $c$-norming set in $B_{Y^*}$ for some $c>0$. We can write each $y_n^*$ as
$$
y_n^*=\sum_{\gamma\in\Gamma} \lambda_\gamma^n e_\gamma^*,
$$
where $\sum_{\gamma\in\Gamma} |\lambda_\gamma^n|\leq 1$. Thus, for every $n\in\omega$, the set $S_n=\{\gamma\in\Gamma\::\:\lambda_\gamma^n\neq 0\}$ is countable and, consequently, $S=\bigcup_{n\in\omega} S_n$ is also a countable set. We claim that the set $\{e_\gamma^*\::\:\gamma\in S\}$ is $c$-norming for $Y$. Indeed, let us first note that for every $y\in Y$ we have
$$
|y_n^*(y)|=\biggl|\sum_{\gamma\in S_n}\lambda_\gamma^ne_\gamma^*(y)\biggr|=\sum_{\gamma\in S_n} |\lambda_\gamma^n||e_\gamma^*(y)|\leq \biggl(\sum_{\gamma\in S_n} |\lambda_\gamma^n|\biggr)\sup_{\gamma\in S_n} |e_\gamma^*(y)|\leq \sup_{\gamma\in S_n} |e_\gamma^*(y)|.
$$
From the latter we deduce that
$$
 c\,\|y\|\leq \sup_{n\in\omega} |y_n^*(y)| \leq \sup_{n\in\omega} \sup_{\gamma\in S_n} |e_\gamma^*(y)|\leq \sup_{\gamma\in S}|e_\gamma^*(y)| \quad \text{ for all } y\in Y.
$$
Finally, observe that $Y$ is lattice embeddable into $\ell_\infty(S)$ through the lattice embedding given by $y\mapsto \bigl(e_\gamma^*(y)\bigr)_{\gamma\in S}$.
\end{proof}

The preceding proposition motivates the following definition:

\begin{defin}\label{def:desired-property}
    We say that a Banach space $X$ has the \textbf{Desired Property (DP)} if  for every norming sequence $(e_n^*)_{n\in \omega}$ in $X^*$ there exists an element $f \in X$ such that no element $g \in X$ satisfies 
$$ e_n^*(g)=|e_n^*(f)| \mbox{ for every } n \in \omega.$$
\end{defin}

We will see next that a Banach space has the (DP) if and only if it is not isomorphic to a sublattice of $\ell_\infty$. Therefore, by Proposition 4.1, in order to prove that $\PS$ is not isomorphic to a Banach lattice it will be sufficient to check that this space has the (DP).

\begin{cor}\label{cor:desired property}
Given an isomorphic predual $X$ of $\ell_1(\Gamma)$ which has a countable norming set, the following statements are equivalent:
\begin{enumerate}
    \item $X$ is isomorphic to a Banach lattice.
    \item $X$ is isomorphic to a sublattice of $\ell_\infty$.
    \item $X$ does not have the (DP). That is, there exists a norming sequence $(x_n^*)_{n\in\omega}$ in $B_{X^*}$ such that for every $f\in X$ there is an element $g\in X$ such that
    $$
    x_n^*(g)=|x_n^*(f)|, \text{ for every } n\in \omega.
    $$
\end{enumerate}
\end{cor}
\begin{proof}
(1) $\Leftrightarrow$ (2) is just Proposition 4.1.

(2) $\Rightarrow$ (3). Let $T\colon X\to Y$ be an invertible operator onto a sublattice $Y$ of $\ell_\infty$ and let $C=\|T\|\|T^{-1}\|$. If we denote by $(e_n^*)_{n\in\omega}\subseteq \ell_\infty^*$ the canonical basis of $\ell_1$, the natural order in $\ell_\infty$ is given by
$$
f\leq g \quad \text{ if and only if } \quad e_n^*(f)\leq e_n^*(g), \quad \text{for every } n\in\omega.
$$
It is clear that $(e_n^*)_{n\in\omega}$ is $1$-norming in $\ell_\infty$ and, hence, the sequence of restrictions
$$
y_n^*:=\left.e_n^*\right|_Y, \quad n\in\omega
$$
is $1$-norming in $Y$. Now, define
$$
x_n^*:=\frac{1}{\|T\|}T^*y_n^*,\quad n\in\omega.
$$
It is straightforward to check that $(x_n^*)_{n\in\omega}\subseteq B_{X^*}$ is $1/C$-norming in $X$ and, given $f\in X$, if we take $g=T^{-1}|Tf|$, then for every $n\in\omega$ we have
$$
x_n^*(g)=\frac{1}{\|T\|}T^*y_n^*\bigl(T^{-1}|Tf|\bigr)=\frac{1}{\|T\|}y_n^*\bigl(|Tf|\bigr)=|x_n^*(f)|.
$$

(3) $\Rightarrow$ (2). Suppose that $X$ fails the (DP). That is, there exists a norming sequence $(x_n^*)_{n\in\omega}\subseteq B_{X^*}$ such that for every $f\in X$ there is an element $g\in X$ such that $x_n^*(g)=\left|x_n^*(f)\right|$ for every $n\in\omega$. We define the following mapping
$$
\begin{array}{ccc}
    T: & X\longrightarrow & \ell_\infty \\
      &f \longmapsto & \bigl(x_n^*(f)\bigr)_{n\in\omega},
\end{array}    
$$
which is clearly linear and bounded below (since $(x_n^*)_{n\in\omega}$ is norming). Hence, $Y=T(X)$ is a closed subspace of $\ell_\infty$. Moreover, it is a sublattice. Indeed, by hypothesis, for any $\bigl(x_n^*(f)\bigr)_{n\in\omega}\in Y$ there exists $g\in X$ such that $\bigl(x_n^*(g)\bigr)_{n\in\omega}=\bigl(|x_n^*(f)|\bigr)_{n\in\omega}=\bigl|\bigl(x_n^*(f)\bigr)_{n\in\omega}\bigr|$; that is, the absolute value of $\bigl(x_n^*(f)\bigr)_{n\in\omega}$ also belongs to $Y$.
\end{proof}

\begin{thm} \label{thm:main}
$\PS$ is not isomorphic to a Banach lattice.
\end{thm}
\begin{proof}
We will prove this fact by showing that $\PS$ has the (DP). Fix a norming sequence $(e_n^*)_{n \in \omega}$ in $B_{\PS^*}$. Our aim is to find an $f\in \PS$ such that no $g\in \PS$ satisfies 
\begin{equation}\label{eq:non-existence-absolute-value}
\langle e_n^*, g\rangle=|\langle e_n^*, f\rangle|, \quad\text{ for every } n\in\omega.  
\end{equation}
The very definition of $\PS$ allows us to write, for every $n\in\omega$, $e_n^*=\mu_n+\bar\nu_n$, where $\mu_n\in \ell_1(\omega\times 2)$ and $\bar\nu_n\in \ell_1(\fc\times 2)$ such that $\|\mu_n\|_1+\|\bar\nu_n\|_1\leq 1$ and $\mu_n(k,0)=-\mu_n(k,1)$ for every $k\in \omega$ and $\bar\nu_n(\alpha,0)=-\bar\nu_n(\alpha,1)$ for every $\alpha<\fc$ --since we are identifying $\PS^*$ with $\JL(\mathcal{A})^\perp$, see Section \ref{subsec:norming} for more details. Given that every element of $(\bar\nu_n)_{n\in\omega}$ vanishes on finite subsets of $\omega\times 2$, we have
$$
e_n^*(f_k)=e_n^*\left(1_{(k,0)}-1_{(k,1)}\right)=\mu_n(k,0)-\mu_n(k,1)=2\,\mu_n(k,0),\,\, \text{ for all } k,n\in\omega,
$$
where $f_k$ had already been defined in equation (\ref{eq:fn's}).

In addition, as $(e_n^*)_{n\in\omega}$ is a norming set, there exists $\tilde{c}>0$ such that
$$
2\,\sup_n|\mu_n(k,0)|=\sup_n|e_n^*(f_k)|\geq \tilde{c},\quad \text{ for every } k\in\omega.
$$
Since $\sup_n|\mu_n(k,0)|>\frac{\tilde{c}}{4}$ for every $k\in\omega$, there exists a function $\pi:\omega\to\omega$ such that $\left|\mu_{\pi(k)}(k,0)\right|>\frac{\tilde{c}}{4}$ for every $k\in\omega$. Moreover, as $\|\mu_n\|_1\leq 1$ for all $n\in\omega$, it follows that the set $\pi^{-1}(n)$ must be finite for every $n\in\omega$. Therefore, we can find an infinite subset $\omega_0\subseteq\omega$ such that $\left.\pi\right|_{\omega_0}$ is injective. Consequently, $(e_{\pi(n)}^*)_{n\in\omega_0}=(\mu_{\pi(n)},\bar\nu_{\pi(n)})_{n\in\omega_0}$ is a sequence of the form described in Remark \ref{rem:sequences-to-code}.
Thus, there exists $\xi<\mathfrak{c}$ such that $(e_{\pi(n)}^*)_{n\in\omega_0}= (\nu_n^\xi)_{n\in\omega}$, with the notation of Section \ref{subsec:heart}. Recall that, by the way the enumeration has been carried out, we have $\bar\nu_n^\xi(\alpha,i)=0$ whenever $\alpha\geq \xi$, $i\in \{0,1\}$ and $n\in\omega$. Moreover,  by virtue of Remark \ref{rem:final}, the pairs of measures 
\begin{itemize}
    \item $\{\nu_n^\xi: n\in J_2^\xi\}$ and $\{\nu_n^\xi: n\in J_1^\xi\setminus J_2^\xi\}$, 
    \item $\{\nu_n^\xi: n\in J_1^\xi\}$ and $\{\nu_n^\xi: n\in J_0^\xi\setminus J_1^\xi\}$
\end{itemize}
are not $\mathfrak{B}(\fc\setminus\{\xi\})$-separated.

\par For the rest of the proof, we will drop the superindex $\xi$, and simply write $\nu_n,\,\mu_n$ and $\bar\nu_n$ for the measures $\nu_n^\xi,\,\mu_n^\xi$ and $\bar\nu^\xi_n$, respectively --we will also denote $p_n$ instead of $p_n^\xi$. Now, consider the function
\begin{equation}\label{eq:main-thm-1}
   f=1_{B^0_\xi}-1_{B^1_\xi}\in \PS. 
\end{equation}
Since $f$ is supported in $B^0_\xi \cup B^1_\xi$, Remark \ref{rem:zero} asserts that $\bil{\nu_n}{f} =  \bil{\mu_n}{f}$ for every $n\in \omega$. Let us suppose that there exists an element $g\in \PS$ such that
\[
\langle \nu_n, g\rangle=|\langle \nu_n, f\rangle|=|\langle\mu_n,f\rangle| \text{ for every } n\in\omega.\]
We will arrive at a contradiction with the separation of the above pairs of sets. Note that the latter means that the function $f$ defined in (\ref{eq:main-thm-1}) cannot have an absolute value in $\PS$ with respect to the sequence $(\nu_n)_{n\in\omega}=(e_{\pi(n)}^*)_{n\in\omega_0}$; in particular, $f$ cannot have an absolute value with respect to $(e_n^*)_{n\in\omega}$ --which is what we were looking for in (\ref{eq:non-existence-absolute-value}). Our argument closely follows that of \cite[Lemma 5.3]{PSA}. First, we introduce some notation: given $a,b\in \R$ and $\delta>0$, we write $a\approx_\delta b$ to mean $|a-b|<\delta$. 

Let us pick $\delta>0$ satisfying property (P1). Since the subspace of $\PS$ consisting of all simple $\fB$-measurable functions in $\PS$ is dense, there is such a function $h\in \PS$ such that $\|g-h\|<\delta$. Therefore,
$$
|\langle\mu_n,f\rangle|=|\langle \nu_n,f\rangle|\approx_\delta \langle \nu_n,h\rangle, \text{ for every } n\in\omega.
$$

Without loss of generality, we assume that $h=r f+s$, where $r\in\mathbb{R}$ and $s$ is a simple $\mathfrak{B}(\mathfrak{c}\setminus\{\xi\})$-measurable function lying in $\PS$. Let us further suppose that $r\geq 0$. Otherwise, we may apply our argument to the function $ -g$ instead of $g$; that is, if we show that $-|f|$ cannot exist, then neither can $|f|$. Hence, we have:
\begin{equation}\label{eq:main-thm-2}
    |\langle \mu_n, f\rangle|\approx_\delta r\langle \mu_n, f\rangle+\langle \nu_n,s\rangle, \text{ for every } n\in\omega.
\end{equation}

Now, observe that properties (P1) and (P2), together with the definition of $B^0_\xi$ and $B^1_\xi$, yield, for every $n\in J_0^\xi$, 
$$
\begin{array}{cl}
  \bil{\mu_n}{f}\approx_\delta \int_{c_n} f d\mu_n=f(p_n)\mu_n(p_n)+f(c_n\setminus p_n)\mu_n(c_n\setminus p_n) 
      \approx_{2\delta} \left\{\begin{array}{rl}
       2a   &  \text{ if } n\in J_2^\xi,\\[1mm]
       -2a   &  \text{ if } n\in J_1^\xi\setminus J_2^\xi, \\[1mm]
       0 &  \text{ if } n\in J_0^\xi\setminus J_1^\xi.
     \end{array}\right.
\end{array}
$$
Hence,
\begin{equation}\label{eq:main-thm-3}
   \langle \mu_n, f\rangle    \approx_{3\delta} \left\{\begin{array}{rl}
       2a   &  \text{ if } n\in J_2^\xi,\\[1mm]
       -2a   &  \text{ if } n\in J_1^\xi\setminus J_2^\xi, \\[1mm]
       0 &  \text{ if } n\in J_0^\xi\setminus J_1^\xi.
     \end{array}\right.   
\end{equation}
We deduce from the previous equation that
\begin{equation}\label{eq:main-thm-4}
    |\langle \mu_n,f\rangle | \approx_{3\delta} \left\{\begin{array}{rl}
       2a   &  \text{ if } n\in J_1^\xi,\\[1mm]
       0 & \text{ if } n\in J_0^\xi \setminus J_1^\xi.
     \end{array}\right.   
\end{equation}
Finally, using (\ref{eq:main-thm-3}) and (\ref{eq:main-thm-4}), we infer from (\ref{eq:main-thm-2}) the following relations:
\begin{equation}\label{eq:main-thm-5}
\left\{\begin{array}{rrl}
       2a\approx_{\delta(4+3r)}& 2ra +\langle \nu_n,s\rangle  &  \text{ if } n\in J_2^\xi,\\[1mm]
       2a\approx_{\delta(4+3r)}& -2ra+\langle \nu_n,s\rangle  &  \text{ if } n\in J_1^\xi\setminus J_2^\xi,\\[1mm]
       0\approx_{\delta(4+3r)}& \langle \nu_n,s\rangle & \text{ if } n\in J_0^\xi \setminus J_1^\xi.
     \end{array}\right.   
\end{equation}

First, suppose that $0\leq r\leq 1/2$. The first two relations of (\ref{eq:main-thm-5}) give, for every $n\in J_1^\xi$,
$$
\langle \nu_n,s\rangle\geq 2(1-r)a-\delta (4+3r)\geq a-\frac{11}{2}\delta,
$$
while the third one gives, for every $k\in J_0^\xi \setminus J_1^\xi$,
$$
\langle \nu_k,s\rangle\leq \delta (4+3r)\leq \frac{11}{2}\delta.
$$
Thus, for any $n\in J_1^\xi$ and any $k\in J_0^\xi \setminus J_1^\xi$, we have, using $\delta<c/11$ and $a\geq c$,
$$
\langle \nu_n,s\rangle-\langle \nu_k,s\rangle\geq a-11\delta>0.
$$
This already implies --see Lemma \ref{lem:sep}-- that the sets $\{\nu_n: n\in J_1^\xi\}$ and $\{\nu_n: n\in J_0^\xi\setminus J_1^\xi\}$ are $\fB(\fc\setminus\{\xi\})$-separated.
On the other hand, if $r\geq 1/2$, then using relations (\ref{eq:main-thm-5}) again, we infer that for every $n\in J_1^\xi \setminus J_2^\xi$ and every $k\in J_2^\xi$

\begin{align*}
\langle \nu_n,s\rangle -\langle \nu_k,s\rangle &  \geq 2a(1+r)-\delta(4+3r) -\bigl(2a(1-r)+ \delta(4+3r)\bigr) 
\\
     & = 2r(2a-3\delta)-8\delta \geq 2a-11\delta>0.
\end{align*}
Hence, the sets $\{\nu_n: n\in J_2^\xi\}$ and $\{\nu_n: n\in J_1^\xi\setminus J_2^\xi\}$ are $\mathfrak{B}(\mathfrak{c}\setminus\{\xi\})$-separated, again by Lemma \ref{lem:sep}. Thus, in both cases we arrive at a contradiction.
\end{proof}

\subsection{Relation to other classes of $\cL_\infty$-spaces}
Apart from the class of AM-spaces, other well-known classes of $\cL_\infty$-spaces are those of G-spaces and $C_\sigma(K)$-spaces. \emph{G-spaces} can be characterized as the closed subspaces $X$ of some $C(K)$ for which there exists a certain set of triples 
$A=\{(t_\alpha, t'_\alpha, \lambda_\alpha): \alpha \in \Gamma\} \subseteq K \times K \times \mathbb R $ so that $X = \{f\in C(K): f(t_\alpha) = \lambda_\alpha f(t'_\alpha) \ \text{ for all } \alpha\in \Gamma\}$. On the other hand, \emph{$C_\sigma(K)$-spaces} are the closed subspaces $X$ of some $C(K)$ which are of the form $$X=\{f \in C(K): f(\sigma t) = -f(t) \ \text{ for all } t\in K\},$$ where $\sigma: K \to K$ is a homeomorphism with $\sigma^2 = \operatorname{Id}$. 

\par These classes are also rather natural in the context of $1$-complemented subspaces. Indeed, $G$-spaces are precisely those Banach spaces which are $1$-complemented in some AM-space, while a Banach space is $1$-complemented in a $C(K)$-space if and only if it is a $C_\sigma(K)$-space \cite[Theorem 3]{LW}.

\par It is clear that $C_\sigma(K)$-spaces are $G$-spaces. On the other hand, Kakutani's representation theorem asserts that every AM-space $X$ is of the form 
$X = \{f\in C(K): f(t_\alpha) = \lambda_\alpha f(t'_\alpha) \ \text{ for all }  \alpha\in \Gamma\}$
for some compact space $K$ and a certain set of triples $A=\{(t_\alpha, t'_\alpha, \lambda_\alpha)\::\:\alpha\in \Gamma\}\subseteq K \times K \times \mathbb [0,\infty)$.
Therefore, AM-spaces are in particular $G$-spaces. The properties of $\PS$ and Theorem \ref{thm:main} show that the latter are a strictly larger class: 

\begin{cor} There is a $C_\sigma(K)$-space which is not isomorphic to an AM-space. In particular, there is a $G$-space which is not isomorphic to an AM-space.
\end{cor}

\par In fact, equation (\ref{eq:Csigma}) witnesses 
$\PS$ is a $C_\sigma(K)$-space for $K=K_\cB$: we can write
\[ \PS = \{f\in C(K_\cB): f(\sigma p) = -f(p) \, \text{ for all }\, p\in K_\cB \}, \]
where $\sigma: K_\cB \to K_\cB$ is defined as
\[ \sigma(n,i)=(n,1-i), \quad\quad \sigma(p_{B_\xi^i}) = p_{B_\xi^{i-1}}, \quad\quad  \sigma(\infty)=\infty, \] 
for $n\in \omega$, $i\in\{0,1\}$ and $\xi<\fc$. This yields two interesting consequences. First, let us observe that such map $\sigma$ has $\infty$ as its only fixed point. It is shown in \cite[Corollary of Theorem 10, Section 10]{Lacey-book} that every $C_\sigma(K)$-space, where $K$ is a scattered compact space and $\sigma$ has no fixed points, is isometrically isomorphic to a $C(K)$-space. This result is no longer true if $\sigma$ has only one fixed point, as the existence of $\PS$ shows. 

On the other hand, in \cite[Theorem 7]{HHM86} it is shown that a Banach space $X$ has an ultrapower isometric to an ultrapower of $c_0$ if and only if $X$ is isometric to a $C_\sigma(K)$-space for  $K$ a totally disconnected compact space with a dense subset of isolated points and such that $\sigma$ has a unique fixed point which is not isolated in $K$. We therefore conclude that $(\PS)^\mathcal{U}$ is isometric to $(c_0)^\mathcal{U}$ for some ultrafilter $\mathcal{U}$. This should be compared with \cite[Theorem 4.1]{HHM83}, in which the authors construct a Banach space $X$ such that $X^\mathcal{U}$ is isometric to $(c_0)^\mathcal{U}$ for some ultrafilter $\mathcal{U}$, but $X$ is not \textit{isometric} to any Banach lattice.

\section{The CSP for complex Banach lattices}
\label{SectionCSPcomplex}

In this section, we will show how $\PS$ can be \textit{modified} to provide a negative solution to the Complemented Subspace Problem for \textit{complex Banach lattices}. By a complex Banach lattice, as usual, we mean the complexification $X_{\mathbb C}=X\oplus iX$ of a real Banach lattice $X$, equipped with the norm $\|x+iy\|_{X_\mathbb{C}}=\||x+iy|\|_X$, where $|\cdot|:X_\mathbb{C}\to X_+$ is \textit{the modulus} map given by
\begin{equation}\label{eq:norm1}
|x+iy|=\sup_{\theta\in [0,2\pi]}\{x\cos\theta +y\sin \theta\}, \qquad \text{for every } x+iy\in X_\mathbb{C}. 
\end{equation}
Given a real Banach space $E$, $E_\mathbb{C}$ denotes the complexification of the real vector space $E$, $E\oplus i E$, endowed with the norm 
\begin{equation}\label{eq:norm2}
   \|x+iy\|=\sup_{\theta\in [0,2\pi]}\|x\cos \theta+y\sin\theta\|, \quad \text{ for every } x+iy\in E_{\mathbb{C}}. 
\end{equation}
It is not difficult to check that the norm induced by \eqref{eq:norm1} and the one in \eqref{eq:norm2} are equivalent in the class of complex Banach lattices. Moreover, both definitions actually coincide in the complexification of a $C(K)$-space or an $L_p$-space \cite[Section 3.2, Exercises 3 and 5]{AA-book}.

A \textit{complex sublattice} $Z$ of a complex Banach lattice $X_\mathbb{C}$ is the complexification $Y_\mathbb{C}$ of a real sublattice $Y$ of $X$. A $\mathbb{C}$-linear operator between two complex Banach lattices $T:X_\mathbb{C}\to Y_\mathbb{C}$ is said to be a \textit{complex lattice homomorphism} if there exists a lattice homomorphism $S:X\to Y$ such that $T(x_1+ix_2)=Sx_1+iSx_2$ for all $x_1,x_2\in X$. We refer to \cite[Chapter II, Section 11]{Schaefer-book} for further information on complex Banach lattices.

As we mentioned in the introduction, one of the main motivations that led us to consider the complex version of the CSP is the following result of Kalton and Wood \cite{KW}: every $1$-complemented subspace of a complex space with a $1$-unconditional basis has a $1$-unconditional basis. This result is not true in the real case \cite{BFL}, even though it is still unknown whether every complemented subspace of a space with an unconditional basis also has an unconditional basis. Additionally, there are other results for complex Banach lattices which fail in the real setting. Let us recall the following facts:
\begin{itemize}
    \item An M-projection $P$ on $X_\mathbb{C}$ --respectively, an L-projection--, i.e.~a projection satisfying $\|x\|=\max\{\|Px\|,\|x-Px\|\}$ for every $x\in X_\mathbb{C}$ --resp., $\|x\|=\|Px\|+\|x-Px\|$--, is always a band projection \cite{DKL}.
    \item If a complex Banach lattice can be written as $X_\mathbb{C}=E\oplus F$ such that $\|x+y\|=\||x|\lor|y|\|$ for all $x\in E$ and $y\in F$, then $|x|\land |y|=0$ for $x\in E$, $y\in F$; in other words, $E\oplus F$ is a band decomposition of $X_\mathbb{C}$ \cite{Kalton(Dales)}.
\end{itemize}

For the purpose of providing a counterexample to the CSP for complex Banach lattices, a natural approach would be to take $(\PS)_\mathbb{C}=\PS\oplus i\PS$, the complexification of $\PS$, which is one complemented in $\JL(\cB)_{\mathbb{C}}\equiv C_{\mathbb{C}}(K_\mathcal{B})$ --which is a complex Banach lattice. Nevertheless, we do not know whether $(\PS)_\mathbb{C}$ is isomorphic to a complex Banach lattice. Instead, we will show how the construction of $\PS$ can be slightly modified in order to give a negative solution to the CSP in the complex setting as well. This variation of $\PS$, which we denote by $\widetilde{\PS}$, will have the same form as the space $X$ described in Subsection \ref{subsec:general}, so it will also be $1$-complemented in a $C(K)$-space; but $\widetilde{\PS}$ will have the additional feature that its complexification cannot be isomorphic to a complex Banach lattice. In particular, it will not be isomorphic to a real Banach lattice either --see Corollary \ref{cor:final} below.

We start with a complex version of the notion of admissible sequence --cf. Definition \ref{def:admissible-real}.

\begin{defin}\label{def:admissible-complex} We say that a sequence $(\mu_n)_{n\in\omega}$ in the unit ball of $\ell_1(\omega \times 2)\oplus i \ell_1(\omega \times 2)$ is $\mathbb C$-\textit{admissible} if $\mu_n(k,0)=-\mu_n(k,1)$ for every $n,k\in\omega$ and $\inf_{n\in\omega} |\mu_n(n,0)|>0$.
\end{defin}

Note that if $\mu_n(c_k)=0$ for every $n,k\in\omega$, then it follows that $\re \mu_n(c_k)=\im \mu_n(c_k)=0$ for $n,k\in\omega$. Nevertheless, this does not imply that $\inf_{n\in\omega} |\re \mu_n(n,0)|>0$ or $\inf_{n\in\omega} |\im \mu_n(n,0)|>0$, so $(\re \mu_n)_{n\in\omega}$ or $(\im \mu_n)_{n\in\omega}$ are not necessarily \textit{real} admissible sequences. This is an obvious obstruction to directly show that the complexification of $\PS$ cannot be isomorphic to a complex Banach lattice. Instead, working directly with the complex version of the notion of \textit{admissibility} (Definition \ref{def:admissible-complex}) we will show that with small modifications in the construction of $\PS$ one can produce a space $\widetilde{\PS}$ with the following \textit{desired property}.

\begin{defin}\label{def:complex-desired-property}
    We say that a Banach space $X$ has the \textbf{Complex Desired Property ($\mathbb{C}$-DP)} if  for every norming sequence $(e_n^*)_{n\in \omega}$ in $X^*$ there exists an element $f \in X$ such that no element $g \in X$ satisfies 
$$ e_n^*(g)=|e_n^*(f)| \mbox{ for every } n \in \omega.$$
\end{defin}

As we have already mentioned, the new space $\widetilde{\PS}$ does have the same shape as the space $X$ explained in Subsection \ref{subsec:general}. Following the same notation, $\widetilde{\PS}$ is therefore the range of the contractive projection $Q=\text{Id}_{\JL(\cB)}-P$, whereas $\JL(\cA)$ is $1$-complemented in $\JL(\cB)$ by $P$. Note that the operator $Q_\mathbb{C}:\JL(\cB)_\mathbb{C}\to \JL(\cB)_\mathbb{C}$ defined by $Q_\mathbb{C}(f_1+if_2):=Qf_1+iQf_2$ is a norm-one $\mathbb{C}$-linear projection with range $(\widetilde{\PS})_\mathbb{C}$ --see \cite[Lemma 1.7]{AA-book}; similarly, $\JL(\cA)_\mathbb{C}$ is the range of the contractive projection $P_\mathbb{C}$. It should also be noted that now have $(\widetilde{\PS})^*_{\mathbb{C}}\equiv\JL(\cA)_{\mathbb{C}}^\perp$.

Following similar steps as in \cite{PSA}, it is possible to construct two almost disjoint families $\mathcal{A}=\{A_\xi\::\:\xi<\mathfrak{c}\}$ and $\mathcal{B}=\{B_\xi^0,\,B_\xi^1\::\:\xi<\mathfrak{c}\}$ in $\mathcal{P}(\omega\times 2)$ and a suitable enumeration of sequences $(\nu_n^\xi)_{n\in\omega}=(\mu_n^\xi,\bar\nu_n^\xi)_{n\in\omega}$, for $\xi<\fc$, in the unit ball of $\bigl(\ell_1(\omega\times 2)\oplus_1\ell_1(\fc\times 2)\bigr)_\mathbb{C}$ satisfying the properties below:
\begin{itemize}
    \item[\emph{i)}] $(\mu_n^\xi)_{n\in\omega}$ is $\mathbb{C}$-admissible,
    \item[\emph{ii)}]  $\bar\nu_n^\xi(\alpha,0)+\bar\nu_n^\xi(\alpha,1)=0$ for every $\alpha<\fc$,
    \item[\emph{iii)}] $\bar\nu_n^\xi(\alpha,j)=0$ whenever $\alpha\geq \xi$, $n\in\omega$ and $j\in\{0,1\}$;
\end{itemize}
in such a way that if $c:=\inf_{n\in\omega}|\mu_n^\xi(n,0)|$ and $\delta$ represents a fixed number in the interval $(0,c/22)$, there are three infinite sets $J_2^\xi \subseteq J_1^\xi \subseteq J_0^\xi \subseteq \omega$ such that $\omega\setminus J_0^\xi$, $J_0^\xi\setminus J_1^\xi$ and $J_1^\xi\setminus J_2^\xi$ are also infinite, with the following properties:

    \begin{enumerate}
    \item[(Q1)] For every  $n\in J_0^\xi$, $|\mu_n^\xi|\bigl( (J_0^\xi \times 2\bigr)\setminus c_n)<\delta$.
	\item[(Q2)]  There exist $a\in\mathbb{C}$ with $|a|\geq c$ and $p_n^\xi \in \{\{(n,0)\},\{(n,1)\}\} $, such that $|\mu_n^\xi(p_n^\xi)-a|<\delta$ for every $n\in J_1^\xi$ and $\re a\geq \frac{c}{\sqrt{2}}$ or $\im a\geq \frac{c}{\sqrt{2}}$.
    \item[(Q3)] For every $\alpha<\xi$, the pairs
	\begin{itemize}
		\item $\{\nu_n^\alpha: n\in J^\alpha_2\}$ and $\{\nu_n^\alpha: n\in J^\alpha_1\setminus J^\xi_2\}$,
		\item $\{\nu_n^\alpha: n\in J^\alpha_1\}$ and $\{\nu_n^\alpha: n\in J^\alpha_0 \setminus J^\alpha_1\}$
	\end{itemize}
	are not $\fB(\xi \setminus \{\alpha\})$-separated.
\end{enumerate}

Properties (Q1)--(Q3) can be obtained by adjusting lemmata 5.3 and 5.5 from \cite{PSA} to the definition of $\mathbb{C}$-admissibility. 
To avoid cumbersome repetitions, we will not give an explicit proof of these as similar computations will be detailed in the proof of Theorem \ref{thm:complex-CSP}. Let us however sketch the idea of how (Q2) could be verified: Since, for every $\xi<\mathfrak{c}$, $(\nu_n^\xi)_{n\in\omega}=(\mu_n^\xi,\bar\nu_n^\xi)_{n\in\omega}\subseteq B_{(\widetilde{\PS})_{\mathbb{C}}^*}$ and $(\mu_n^\xi)_{n\in\omega}$ is $\mathbb{C}$-admissible, we have
$$
1\geq \bigl|\nu_n^\xi\bigl(1_{(n,0)}-1_{(n,1)}\bigr)\bigr|=2\,|\mu_n^\xi(n,0)|\geq 2c>0, \quad \text{ for every } n\in\omega.
$$
Hence, passing to a subsequence we may assume that $\bigl(\mu_n^\xi(n,0)\bigr)_{n\in\omega}$ converges to some $b\in\mathbb{C}$. As $\inf_{n\in\omega}|\mu_n^\xi(n,0)|=c>0$, then $|b|\geq c$. Thus, $|\re b|\geq \frac{c}{\sqrt{2}}$ or $|\im b|\geq \frac{c}{\sqrt{2}}$. Let us suppose for instance that $|\re b|\geq \frac{c}{\sqrt{2}}$. Since $\mu_n^\xi(c_k)=0$ for every $n,k\in\omega$, in particular, 
$$
\re \mu_n^\xi(n,0)=-\re\mu_n^\xi(n,1)\quad \text{ for every } n\in\omega.
$$
Consequently, for each $n\in\omega$, we can choose $p_n^\xi=\{(n,0)\}$ or $p_n^\xi=\{(n,1)\}$ such that $\re\mu_n^\xi(p_n^\xi)=|\re \mu_n^\xi(n,0)|\geq 0$, so $\re\mu_n^\xi(p_n^\xi)\to |\re b|$. Finally, passing again to a subsequence if necessary, we obtain $\mu_n^\xi(p_n^\xi)\to a$ with $|a|\geq c$ and $\re a=|\re b|\geq \frac{c}{\sqrt{2}}$.

\par Additionally, let us remark that property (Q3) also implies an analogue of Remark \ref{rem:final}, and therefore, for any $\xi<\fc$, the pairs
	\begin{itemize}
		\item $\{\nu_n^\xi: n\in J^\xi_2\}$ and $\{\nu_n^\xi: n\in J^\xi_1\setminus J^\xi_2\}$,
		\item $\{\nu_n^\xi: n\in J^\xi_1\}$ and $\{\nu_n^\xi: n\in J^\xi_0 \setminus J^\xi_1\}$
	\end{itemize}
	are not $\fB(\fc \setminus \{\xi\})$-separated. 
    
\par We now proceed to prove our main result in this section.

\begin{thm}\label{thm:complex-CSP}
$\widetilde{\PS}\oplus i\widetilde{\PS}$ is not isomorphic to a complex Banach lattice.
\end{thm}
\begin{proof}
We will prove this statement by showing that $(\widetilde{\PS})_\mathbb{C}$ does have the ($\mathbb{C}$-DP), which is equivalent to the fact that $(\widetilde{\PS})_\mathbb{C}$ is not isomorphic to a complex Banach lattice --this can be checked with a straightforward adaptation of the proof of Corollary \ref{cor:desired property}.

Fix a norming sequence $(e_n^*)_{n \in \omega}$ in $B_{(\widetilde{\PS})_\mathbb{C}^*}$. Our aim is to find an $f\in (\widetilde{\PS})_\mathbb{C}$ such that no $g\in (\widetilde{\PS})_\mathbb{C}$ satisfies 
$$
\langle e_n^*, g\rangle=|\langle e_n^*, f\rangle|, \text{ for every } n\in\omega.
$$

We shall denote $e_n^*=e_{n,0}^*+ie_{n,1}^*$, where $e_{n,j}^*\in \widetilde{\PS}^*$ for $j=0,1$. The identification $\widetilde{\PS}\equiv\JL(\cA)^\perp$ allows us to write, for every $n\in\omega$ and $j\in\{0,1\}$, $e_{n,j}^*=\mu_{n,j}+\bar\nu_{n,j}$, where $\mu_{n,j}\in \ell_1(\omega\times 2)$ --which is determined by its values on finite sets of $\omega\times 2$-- and $\bar\nu_{n,j}\in \ell_1(\fc\times 2)$ --which vanishes on finite subsets of $\omega\times 2$-- fulfill $\mu_{n,j}(k,0)=-\mu_{n,j}(k,1)$ for every $k\in \omega$ and $\bar\nu_{n,j}(\alpha,0)=-\bar\nu_{n,j}(\alpha,1)$ for every $\alpha<\fc$ --for details, see Section \ref{subsec:norming}.
Moreover, as $(e_n^*)_{n\in\omega}$ is a norming set, there exists $\tilde{c}>0$ such that
$$
2\sup_n|\mu_n(k,0)|=\sup_n|\mu_n(f_k)|=\sup_n|e_n^*(f_k)|\geq \tilde{c},\quad \text{ for every } k\in\omega,
$$
where $f_k=1_{(k,0)}-1_{(k,1)}$. Arguing in the same way as we did in the proof of Theorem \ref{thm:main}, we can find an injective map $\pi:\omega_0\to \omega$ for some infinite $\omega_0\subseteq\omega$ in such a way that $(\mu_{\pi(n)})_{n\in{\omega_0}}$ is a $\mathbb{C}$-admissible sequence. Therefore, the sequence $(e_{\pi(n)}^*)_{n\in\omega_0}$ is coded by some $(\nu_n^\xi)_{n\in\omega}=(\mu_n^\xi,\bar\nu_n^\xi)_{n\in\omega}$ --where $\xi<\fc$-- satisfying properties \emph{i)}--\emph{iii)} mentioned in the previous comments to this theorem. Recall that the enumeration was chosen so as to satisfy $\bar\nu_n^\xi(\alpha,j)=0$ for all $\alpha\geq \xi$, $n\in\omega$ and $j\in\{0,1\}$.

Again, for the sake of simplicity, for the remainder of the proof of the theorem we will omit the superscript $\xi$ in $\nu_n^\xi$, $\mu_n^\xi$, $\bar\nu_n^\xi$ and $p_n^\xi$. By (Q3), the pairs of measures 
\begin{itemize}
    \item $\{\nu_n: n\in J_2^\xi\}$ and $\{\nu_n: n\in J_1^\xi\setminus J_2^\xi\}$, 
    \item $\{\nu_n: n\in J_1^\xi\}$ and $\{\nu_n: n\in J_0^\xi\setminus J_1^\xi\}$
\end{itemize}
are not $\mathfrak{B}(\fc\setminus\{\xi\})$-separated. Let $c=\inf_{n\in\omega}|\mu_n(n,0)|$ and let $a$ be the complex number appearing in property (Q2). We will only consider the case when $\re a\geq \frac{c}{\sqrt{2}}$, but the proof can be easily adapted to the case when $\im a\geq \frac{c}{\sqrt{2}}$

Let us consider the function
\begin{equation}\label{eq:main-thm-1c}
   f=1_{B^0_\xi}-1_{B^1_\xi}\in (\widetilde{\PS})_\mathbb{C}. 
\end{equation}
We will check that there cannot exist a function in $(\widetilde{\PS})_\mathbb{C}$ giving the modulus of $f$ with respect to the sequence $(\nu_n)_{n\in\omega}=(e_{\pi(n)}^*)_{n\in\omega_0}$; in particular, this would imply that $f$ cannot have a modulus with respect to $(e_{n}^*)_{n\in\omega}$.
Let us suppose that there exists an element $g\in (\widetilde{\PS})_\mathbb{C}$ such that
$$
\langle \nu_n, g\rangle=|\langle \nu_n, f\rangle|=|\langle\mu_n,f\rangle|, \text{ for every } n\in\omega,
$$
where in the second equality we are using that $\bar\nu_n(\xi,j)=0$ for $n\in\omega$ and $j\in\{0,1\}$ --see Remark \ref{rem:zero}. We will arrive at a contradiction with the separation of the two pairs of sets of measures defined above. The computations will be very similar to those performed in the proof of Theorem \ref{thm:main} --following again very closely the argument of \cite[Lemma 5.3]{PSA}. We will keep this notation: given $a,b\in \R$ and $\delta>0$, we write $a\approx_\delta b$ to mean $|a-b|<\delta$. 

Let us fix $\delta>0$ as in (Q1)--(Q3). Since the subspace of $(\widetilde{\PS})_\mathbb{C}$ consisting of simple $\fB$-measurable functions is dense in $(\widetilde{\PS})_\mathbb{C}$, there is such a function $h\in (\widetilde{\PS})_\mathbb{C}$ such that $\|g-h\|<\delta$. Therefore,
$$
|\langle\mu_n,f\rangle|=\langle\nu_n,g\rangle\approx_\delta \langle \nu_n,h\rangle, \text{ for every } n\in\omega.
$$

Without loss of generality, we assume that $h=r f+s$, where $r\in\mathbb{C}$ and $s$ is a simple $\mathfrak{B}(\mathfrak{c}\setminus\{\xi\})$-measurable function lying in $(\widetilde{\PS})_\mathbb{C}$. Let us further suppose that $r\geq 0$; if $r=|r|e^{i\theta}$ we may apply our argument to the function $e^{i\theta}g$ instead of $g$. That is, if we prove that $e^{-i\theta}|f|$ cannot exist, then $|f|$ does not exist either. Hence, we have:
\begin{equation}\label{eq:main-thm-2c}
    |\langle \mu_n, f\rangle|\approx_\delta r\langle \mu_n, f\rangle+\langle \nu_n,s\rangle, \text{ for every } n\in\omega.
\end{equation}

Now, observe that properties (Q1) and (Q2), together with the definition of $B^0_\xi$ and $B^1_\xi$, yield, for every $n\in J_0^\xi$
$$
\begin{array}{cl}
  \bil{\mu_n}{f}\approx_{\delta} \int_{c_n} f d\mu_n=f(p_n)\mu_n(p_n)+f(c_n\setminus p_n)\mu_n(c_n\setminus p_n) 
      \approx_{3\delta} \left\{\begin{array}{rl}
       2a   &  \text{ if } n\in J_2^\xi,\\[1mm]
       -2a   &  \text{ if } n\in J_1^\xi\setminus J_2^\xi, \\[1mm]
       0 &  \text{ if } n\in J_0^\xi\setminus J_1^\xi.
     \end{array}\right.
\end{array}
$$
Hence,
\begin{equation}\label{eq:main-thm-3c}
   \langle \mu_n, f\rangle    \approx_{3\delta} \left\{\begin{array}{rl}
       2a    &  \text{ if } n\in J_2^\xi,\\[1mm]
       -2a  &  \text{ if } n\in J_1^\xi\setminus J_2^\xi, \\[1mm]
       0 &  \text{ if } n\in J_0^\xi\setminus J_1^\xi.
     \end{array}\right.   
\end{equation}

We deduce from the previous equation that
\begin{equation}\label{eq:main-thm-4c}
    |\langle \mu_n,f\rangle | \approx_{3\delta} \left\{\begin{array}{rl}
       2|a|   &  \text{ if } n\in J_1^\xi,\\[1mm]
       0 & \text{ if } n\in J_0^\xi \setminus J_1^\xi.
     \end{array}\right.   
\end{equation}

Now, using (\ref{eq:main-thm-3c}) and (\ref{eq:main-thm-4c}), we infer from (\ref{eq:main-thm-2c}) the following relations:
\begin{equation}\label{eq:main-thm-5c}
\left\{\begin{array}{rrl}
       2|a|\approx_{\delta(4+3r)}& 2ra +\langle \nu_n,s\rangle  &  \text{ if } n\in J_2^\xi,\\[1mm]
       2|a|\approx_{\delta(4+3r)}& -2ra+\langle \nu_n,s\rangle  &  \text{ if } n\in J_1^\xi\setminus J_2^\xi,\\[1mm]
       0\approx_{\delta(4+3r)}& \langle \nu_n,s\rangle & \text{ if } n\in J_0^\xi \setminus J_1^\xi.
     \end{array}\right.   
\end{equation}

First, suppose that $0\leq r\leq 1/2$. The first two relations of (\ref{eq:main-thm-5c}) give for every $n\in J_1^\xi$
$$
|\langle \nu_n,s\rangle|\geq 2|1\pm re^{i\alpha}||a|-\delta (4+3r)\geq |a|-\frac{11}{2}\delta,
$$
where $a=|a|e^{i\alpha}$, while the third one gives for every $k\in J_0^\xi \setminus J_1^\xi$,
$$
|\langle \nu_k,s\rangle|\leq \delta (4+3r)=\frac{11}{2}\delta.
$$
Thus, for any $n\in J_1^\xi$ and any $k\in J_0^\xi \setminus J_1^\xi$, we have, using that $\delta<c/22$ and $|a|\geq c$,
$$
|\langle \nu_n,s\rangle|-|\langle \nu_k,s\rangle|\geq |a|-11\delta>0.
$$
This already implies --by adjusting Lemma \ref{lem:sep} to the complex setting-- that the sets $\{\nu_n: n\in J_1^\xi\}$ and $\{\nu_n: n\in J_0^\xi\setminus J_1^\xi\}$ are $\fB(\fc\setminus\{\xi\})$-separated.
On the other hand, if $r\geq 1/2$, then using relations (\ref{eq:main-thm-5c}) again, we infer that for every $n\in J_1^\xi \setminus J_2^\xi$ and every $k\in J_2^\xi$

\begin{align*}
\re \langle \nu_n,s\rangle -\re \langle \nu_k,s\rangle &  \geq 2r\re a+2|a|-\delta(4+3r)-\bigl(\delta(4+3r)+2|a|-2r\re a \bigr) = 
\\
     & = 2r(2\re a-3\delta)-8\delta \geq 2\re a-11\delta>0,
\end{align*}
since we have supposed that $\re a\geq \frac{c}{\sqrt{2}}$. It is clear that if $\im a\geq \frac{c}{\sqrt{2}}$, we may obtain, using the same procedure as shown above, that $\im \langle \nu_n,s\rangle -\im \langle \nu_k,s\rangle\geq 2\re a-11\delta>0$ for every $n\in J_1^\xi \setminus J_2^\xi$ and every $k\in J_2^\xi$. Hence, the sets $\{\nu_n: n\in J_2^\xi\}$ and $\{\nu_n: n\in J_1^\xi\setminus J_2^\xi\}$ are $\mathfrak{B}(\mathfrak{c}\setminus\{\xi\})$-separated. This is a contradiction.
\end{proof}

We have already remarked that $\widetilde{\PS}$ is $1$-complemented in a $C(K)$-space --see the paragraph that comes after Definition \ref{def:complex-desired-property}. We will see below that it is easy to deduce from our last result that $\widetilde{\PS}$ cannot be isomorphic to a Banach lattice. Therefore this modification of $\PS$ is also a counterexample to the CSP for (real) Banach lattices.

\begin{cor}\label{cor:final}
$\widetilde{\PS}$ is not isomorphic to a Banach lattice.
\end{cor}
\begin{proof}
Suppose that there exist a Banach lattice $X$ and an isomorphism $T:\widetilde{\PS}\to X$. Recall that since $(\widetilde{\PS})_\mathbb{C}$ is a subspace of $\JL(\cB)_\mathbb{C}\equiv C(K_\cB)_\mathbb{C}$ and in this space the complex Banach lattice norm induced by (\ref{eq:norm1}) coincides with the one defined in (\ref{eq:norm2}), then the norm of $(\widetilde{\PS})_\mathbb{C}$ is also given by $\|f_1+if_2\|_{(\widetilde{\PS})_\mathbb{C}}=\sup_{\theta\in [0,2\pi]}\|f_1\cos\theta+f_2\sin\theta\|$.

The operator $T_\mathbb{C}:(\widetilde{\PS})_\mathbb{C}\to X_\mathbb{C}$ given by $T_\mathbb{C}(f_1+if_2):=Tf_1+iTf_2$ is clearly $\mathbb{C}$-linear and bijective. Let us now check that $T_\mathbb{C}$ is continuous. First note that by definition of the modulus map (\ref{eq:norm1}), for every $f_1,f_2\in \widetilde{\PS}$ we have $|Tf_1+iTf_2|\leq |Tf_1|+|Tf_2|$. Therefore, for every $f_1+if_2\in (\widetilde{\PS})_\mathbb{C}$ we have
\begin{eqnarray*}
\|T_\mathbb{C}(f_1+if_2)\|_{X_\mathbb{C}}&=&\||Tf_1+iTf_2|\|_X\leq   \|Tf_1\|_X+\|Tf_2\|_X \leq \|T\|\bigl(\|f_1\|_{\widetilde{\PS}}+\|f_2\|_{\widetilde{\PS}}\bigr) \\
&\leq& 2\|T\|\|f_1+if_2\|_{(\widetilde{\PS})_\mathbb{C}},
\end{eqnarray*}
and by the bounded inverse theorem it follows that $T_\mathbb{C}$ is an isomorphism. This is a contradiction with the previous theorem, so $\widetilde{\PS}$ cannot be isomorphic to a Banach lattice.
\end{proof}

\begin{rem}
    Regarding the CSP for complex Banach lattices, one question that remains open is whether every complemented subspace of a complex Banach lattice is linearly isomorphic to the complexification of some real Banach space. Note that if a complex Banach space is the complexification of a real Banach space then, in particular, it is isomorphic to its complex conjugate. As far as we are concerned, all the known examples of complex Banach spaces non-isomorphic to their corresponding complex conjugates --see, for instance, \cite{Anisca, Bourgain, Ferenczi, Kalton:95}-- fail GL-lust, so they cannot be complemented subspaces of complex Banach lattices.
\end{rem}

\section*{Acknowledgements}
We wish to thank Antonio Avil\'es and F\'elix Cabello for interesting discussions related to the topic of this paper. We would also like to thank the anonymous reviewers for their valuable comments and suggestions.

Research of D. de Hevia and P. Tradacete partially supported by grants PID2020-116398GB-I00, CEX2019-000904-S and CEX2023-001347-S funded by  MCIN/AEI/10. 13039/501100011033. D. de Hevia benefited from an FPU Grant FPU20/03334 from the Ministerio de Universidades. 

G. Mart\'inez-Cervantes was partially supported by Fundación Séneca - ACyT Región de Murcia (21955/PI/22), by Generalitat Valenciana project CIGE/2022/97 and by Agencia Estatal de Investigación and EDRF/FEDER “A way of making Europe"
(MCIN/ AEI/10.13039/501100011033) (PID2021-122126NB-C32).

Research of A. Salguero-Alarcón has been supported by projects PID-2019-103961GB-C21 and PID-2023-146505NB-C21 funded by MCIN/AEI/10.13039/501100011033, and IB20038 (Junta de Extremadura).

P. Tradacete is also supported by a 2022 Leonardo Grant for Researchers and Cultural Creators, BBVA Foundation.


\begin{thebibliography}{100}

\bibitem{AA-book} Y. A. Abramovich and C. D. Aliprantis, \emph{An invitation to operator theory.} Grad. Studies in Math. \textbf{50}, American Mathematical Society, Providence, RI, 2002.

 \bibitem{AW75} Y.~Abramovich and P.~Wojtaszczyk, \emph{The uniqueness of order in the spaces $L_p$(0, 1) and $\ell_p$ ($1 \leq p \leq \infty$)}. Mat. Zametki \textbf{18} (1975), 313--325.


\bibitem{AB}
C.~D. Aliprantis and O.~Burkinshaw, \emph{Positive operators}. Springer,  Dordrecht, 2006, Reprint of the 1985 original. 

\bibitem{Anisca} 
R.~Anisca, 
\emph{Subspaces of $L_p$ with more than one complex structure.} 
Proc. Amer. Math. Soc. \textbf{131} (2003), no. 9, 2819--2829. 

\bibitem{AMR}
A.~Avilés, G.~Mart\'inez-Cervantes, J.~Rodr\'iguez, \emph{Weak*-sequential properties of Johnson-Lindenstrauss spaces.} 
J.~Funct.~Anal. \textbf{276} (2019), 3051--3066. 


\bibitem{AMRT}
A.~Avilés, G.~Martínez-Cervantes, A.~Rueda Zoca and P.~Tradacete, \emph{Linear versus lattice embeddings between Banach lattices.} 
Adv. Math. \textbf{406} (2022), Article ID 108574, 14 pp. 

\bibitem{ART}
A.~Avil\'es, J.~Rodr\'{\i}guez and P.~Tradacete,
  \textit{The free Banach lattice generated by a Banach space.} J. Funct. Anal. \textbf{274}, no. 10, 2955--2977 (2018).

\bibitem{B}
Y. Benyamini, \textit{Separable $G$ spaces are isomorphic to $C(K)$ spaces.}
Israel J. Math. \textbf{14} (1973), 287--293. 

\bibitem{B78} Y. Benyamini, \textit{An extension theorem for separable Banach spaces.} Israel J. Math. \textbf{29} (1978), no. 1, 24--30.

\bibitem{BFL}
Y. Benyamini, P. Flinn and D.R. Lewis, \textit{A space without $1$-unconditional basis which is $1$-complemented in a space with a $1$-unconditional basis.} Texas Functional Analysis Seminar, 1983–84. Longhorn Notes, University of Texas Press (1984), 145--149.


  \bibitem{Bernau-LaceyLp}
S. J. Bernau and H. E. Lacey, \textit{The range of a contractive projection on an $L_p$-space.}
Pacific J. Math. \textbf{53} (1974), 21--41.
  
\bibitem{BL76}
S. J. Bernau and H. E. Lacey, \emph{A local characterization of Banach lattices with order continuous norm.}
Studia Math. \textbf{58} (1976), 101--128.

\bibitem{Bou} J. Bourgain, \emph{A result on operators on $C[0,1]$.} J. Operator Theory \textbf{3} (1980), no. 2, 275--289.

\bibitem{Bourgain}
J. Bourgain, 
\emph{Real isomorphic complex Banach spaces need not be complex isomorphic.} 
Proc. Amer. Math. Soc. \textbf{96} (1986), no. 2, 221--226.



\bibitem{BD} J. Bourgain and F. Delbaen, \emph{A class of special $\mathcal{L}_\infty$ spaces.} Acta Math. \textbf{145} (1980), 155--176.

\bibitem{BVL}
A. V. Bukhvalov, A. I. Veksler and G. Ya. Lozanovskii, \emph{Banach lattices - some Banach aspects of their theory.}
Russ. Math. Surv. \textbf{34} (1979), no. 2, 159--212. 

\bibitem{CKT} P. G. Casazza, N. J. Kalton and L. Tzafriri, \textit{Decompositions of Banach lattices into direct sums.} Trans. Amer. Math. Soc., \textbf{304} (1987), 771--800.
	

\bibitem{DLOT}
H. G. Dales, N. J. Laustsen, T. Oikhberg and V. G. Troitsky,
\textit{Multi-norms and Banach lattices.} Dissertationes Math. \textbf{524} (2017), 115 pp.

\bibitem{DJT}
J.~Diestel, H.~Jarchow and A.~Tonge, \emph{Absolutely summing operators.}
  Cambridge Studies in Advanced Math. \textbf{43}, Cambridge University
  Press, Cambridge, 1995.


\bibitem{Douglas} R. G. Douglas, \emph{Contractive projections on an $L_1$ space.} Pacific J. Math. \textbf{15} (1965), 443--462.

  \bibitem{DKL}
L. Drewnowski, A. Kami\'{n}ska and P. Lin,
\emph{On Multipliers and L- and M-Projections in Banach Lattices and K\"{o}the Function Spaces.}
J. Math. Anal. Appl. \textbf{206} (1997), no. 1, 83--102.
  

\bibitem{DS}
 N. Dunford and J. T. Schwartz, \emph{Linear operators Part 1.} John Wiley and Sons, 1958.
  	
\bibitem{ES}
P. Enflo and T. W. Starbird, \emph{Subspaces of $L^1$ containing $L^1$.} Studia Math. \textbf{65} (1979), no. 2, 203--225. 


\bibitem{Ferenczi} 
V. Ferenczi, 
\emph{Uniqueness of complex structure and real hereditarily indecomposable Banach spaces.} 
Adv. Math. \textbf{213} (2007), 462--488. 



\bibitem{FJT} T. Figiel, W. B. Johnson and L. Tzafriri, \textit{On Banach Lattices and Spaces Having Local Unconditional Structure, with Applications to Lorentz Function Spaces }, J. Approx. Theory \textbf{13} (1975), 395--412.

\bibitem{Flinn}
P. Flinn, \textit{On a theorem of N. J. Kalton and G. V. Wood concerning 1-complemented subspaces of spaces having an orthonormal basis.} Texas Functional Analysis Seminar, 1983–84. Longhorn Notes, University of Texas Press (1984), 135--144.

\bibitem{GL} Y. Gordon and D. R. Lewis, \emph{Absolutely summing operators and local unconditional structures.} Acta Math. \textbf{133} (1974), 27--48.

\bibitem{GM}
 W.~T.~Gowers and B.~Maurey, \textit{The unconditional basic sequence problem.} J. Amer. Math. Soc. \textbf{6} (1993), no.~4, 851--874.
  

\bibitem{HHM83}
S. Heinrich, C. W. Henson and L. C. Moore, \emph{Elementary equivalence of $L_1$-preduals.} In: Banach Space Theory and its Applications. Lecture Notes in Math. \textbf{991}. Springer, Berlin, Heidelberg (1983), 79--90.

\bibitem{HHM86}
S. Heinrich, C. W. Henson and L. C. Moore, \textit{Elementary Equivalence of $C_\sigma(K)$ Spaces for Totally Disconnected, Compact Hausdorff $K$.} J. of Symbolic Logic \textbf{51} (1986),  no. 1, 135--146. 

\bibitem{J}
R. C. James, \emph{A non-reflexive Banach space isometric with its second conjugate space.} Proc. Nat. Acad. Sci. U.S.A. \textbf{37} (1951), 174--177.

\bibitem{JKS} W.B. Johnson, T. Kania and G. Schechtman, \emph{Closed ideals of operators on and complemented subspaces of Banach spaces of functions with countable support}. Proc. Amer. Math. Soc. \textbf{144} (2016), 4471--4485.

\bibitem{JL1974} W.B. Johnson, J. Lindenstrauss, \emph{Some remarks on weakly compactly generated Banach spaces.} Israel J.~Math. \textbf{17} (1974), 219--230.


\bibitem{JLhandbook} W.B. Johnson, J. Lindenstrauss, \emph{Basic concepts in the geometry of Banach spaces.} In Handbook of the geometry of Banach spaces, Vol. I, 1--84. North-Holland Publishing Co., Amsterdam, 2001.

\bibitem{JLS}
W. B. Johnson, J. Lindenstrauss and G. Schechtman, \emph{On the relation between several notions of unconditional structure.} Israel J. Math. \textbf{37} (1980), no. 1-2, 120--129.

\bibitem{JMST}
W. B. Johnson, B. Maurey, G. Schechtman and L. Tzafriri, \emph{Symmetric structures in Banach spaces.} 
Mem. Am. Math. Soc. \textbf{217} (1979). 




\bibitem{kalton}
N. J. Kalton, \emph{Lattice Structures on Banach Spaces.} Mem. Amer. Math. Soc., \textbf{103} (1993).


\bibitem{Kalton:95} 
N. J. Kalton, 
\emph{An elementary example of a Banach space not isomorphic to its complex conjugate.} 
Canad. Math. Bull. \textbf{38} (1995), no. 2, 218--222.

\bibitem{Kalton(Dales)}
N. J. Kalton,
\emph{Hermitian operators on complex Banach lattices and a problem of Garth Dales}. J. Lond. Math. Soc. (2) \textbf{86} (2012), 641--656.

 \bibitem{KW}
 N. J. Kalton and G. V. Wood, \emph{Orthonormal systems in Banach spaces and their applications}. Math. Proc. Cambridge Philos. Soc. \textbf{79} (1976), 493--510.


\bibitem{KL}
P. Koszmider and N. J. Laustsen, \emph{A Banach space induced by an almost disjoint family, admitting only few operators and decompositions.}  Adv. Math. \textbf{381} (2021), paper ID 107613, 39 pp. 
  
\bibitem{kri}
J. L. Krivine, \emph{Constantes de Grothendieck et fonctions de type positif sur les sphères}.
Adv. Math. \textbf{31} (1979), no. 1, 16--30. 

\bibitem{Lacey-book}
H. E. Lacey, \emph{The isometric theory of classical Banach spaces}. Springer-Verlag, Berlin, 1974.

 \bibitem{LLOT}
 D. H. Leung, L. Li, T. Oikhberg and M. A. Tursi, \emph{Separable universal Banach lattices.} Israel J. Math. \textbf{230} (2019), no. 1, 141--152. 

  \bibitem{LP68}
J. Lindenstrauss and A. Pe{\l{}}czy{\'{n}}ski, \emph{Absolutely summing operators in $\mathcal{L}_p$-spaces and their applications.} Studia Math. \textbf{29} (1968), 275--326.
  
\bibitem{LR69}
 J.~Lindenstrauss and H. P.~Rosenthal, \emph{The $\mathcal{L}_p$-spaces.} Israel J. Math. \textbf{7}
(1969), 325--349.

\bibitem{LT2} J. Lindenstrauss and L. Tzafriri, \emph{Classical Banach spaces II}. Springer-Verlag, Berlin, 1977.

\bibitem{LW} J. Lindenstrauss and D. E. Wulbert, \emph{On the classification of the Banach spaces whose duals are $L_1$ spaces}.
J. Funct. Anal. \textbf{4} (1969), 332--349.


\bibitem{LAT}
J. L\'opez-Abad and P. Tradacete, \emph{Shellable weakly compact subsets of $C[0,1]$.} Math. Ann. \textbf{367} (2017), no. 3-4, 1777--1790.

\bibitem{MP}
W. Marciszewski and G. Plebanek, \emph{Extension operators and twisted sums of $c_0$ and $C(K)$-spaces.} J. Funct. Anal. \textbf{274} (2018), 1491--1529.

\bibitem{Meyer} P.~Meyer-Nieberg, \textit{Banach lattices}, Springer-Verlag, 1991.

\bibitem{Pelczynski}
 A. Pe\l czy\'nski, \emph{Sur certaines propriétés isomorphiques nouvelles des espaces de Banach de fonctions holomorphes $A$ et $H^\infty$.} C. R. Acad. Sci. Paris Sér. A \textbf{279} (1974), 9--12.
 
\bibitem{PW}
 A. Pe\l czy\'nski and M. Wojciechowski,
\emph{Sobolev spaces in several variables in $L^1$-type norms are not isomorphic to Banach lattices.}
Ark. Mat. \textbf{40} (2002), no. 2, 363--382. 
 
\bibitem{PW2}
 A. Pe\l czy\'nski and M. Wojciechowski,
\emph{Sobolev spaces.} Handbook of the geometry of Banach spaces, vol. 2, 1361–1423, North-Holland, Amsterdam, 2003. 
 
\bibitem{PSA1}	G. Plebanek and A. Salguero Alarc\'on,
\textit{On the three space property for $C(K)$-spaces.}
J. Funct. Anal. \textbf{281}
(2021), paper ID 109193, 15 pp.

\bibitem{PSA} G. Plebanek and A. Salguero Alarc\'on, \textit{The complemented subspace problem for $C(K)$-spaces: a counterexample.} 
	Adv. Math. \textbf{426} (2023), 109103, 20 pp.

 \bibitem{Beata}
B. Randrianantoanina, \textit{Norm-one projections in Banach spaces.} Taiwanese J. Math. \textbf{5} (2001), no. 1, 35--95.
	
\bibitem{Ros72} H. P. Rosenthal, \textit{On factors of $C[0, 1]$ with non-separable dual.} Israel J. Math. \textbf{13} (1972), 361--378.

\bibitem{Ros} H. P. Rosenthal,
\textit{The Banach spaces $C(K)$.}  Handbook of the Geometry of Banach spaces, vol. 2, North-Holland, Amsterdam, 2003,  1547--1602.
	
\bibitem{Schaefer-book} H. Schaefer, \emph{Banach Lattices and Positive Operators},  Die Grundlehren der mathematischen Wissenschaften \textbf{215}, Springer-Verlag, Berlin-Heidelberg-New York, 1974.


\bibitem{Tal} M. Talagrand, 
\emph{Some weakly compact operators between Banach lattices do not factor through reflexive Banach lattices.}
Proc. Amer. Math. Soc. \textbf{96} (1986), no. 1, 95--102. 

\bibitem{Tal90} M. Talagrand, \emph{The three-space problem for $L^1$.} J. Amer. Math. Soc. \textbf{3} (1990), no. 1, 9--29. 

\bibitem{Tse} A. Tselishchev, \emph{Absence of Local Unconditional Structure in Spaces of Smooth Functions on Two-Dimensional Torus.} J. Math. Sci. \textbf{261} (2022), 832--843.

\end{thebibliography}
\end{document}